\documentclass{amsart}
\usepackage{comment}
\usepackage{subcaption}
\usepackage[english]{babel}

\usepackage{amsthm}

\usepackage{amssymb}
\usepackage{enumerate, xspace}
\usepackage{dsfont}
\usepackage{amsfonts}
\usepackage{xcolor}
\usepackage{graphicx}
\graphicspath{ {./} }
\usepackage{amsmath}
\usepackage[normalem]{ulem}
\usepackage{graphicx}
\usepackage[T1]{fontenc}
\usepackage[utf8]{inputenc} 
\usepackage{bbm}
\usepackage[colorlinks=true,linkcolor=violet,citecolor=blue,urlcolor=blue]{hyperref}
\usepackage{amsaddr}
\usepackage{bbding}
\usepackage[misc]{ifsym} 

\newtheorem{theorem}{Theorem}

\newtheorem{definition}[theorem]{Definition}

\newtheorem{lemma}[theorem]{Lemma}

\newtheorem{problem}[theorem]{Problem}
\newtheorem{proposition}[theorem]{Proposition}
\newtheorem{remark}[theorem]{Remark}

\newtheorem{maintheorem}{Result}

\newcommand\cL{{\mathcal L}}

\newcommand\cP{{\mathcal P}}

\newcommand\bN{{\mathbb N}}

\newcommand\bR{{\mathbb R}}

\newcommand\fB{{\mathfrak B}}
\newcommand\fH{{\mathfrak H}}

\newcommand{\norm}[1]{\left\Vert#1\right\Vert}

\DeclareMathOperator*{\esssup}{ess\, sup}
\DeclareMathOperator*{\essinf}{ess\, inf}

\title[Optimal response for SDE by local kernel perturbations]{Optimal response for stochastic differential equations by local kernel perturbations}

\author{Gianmarco Del Sarto\textsuperscript{1}, Stefano Galatolo\textsuperscript{2}, Sakshi Jain\textsuperscript{3, \Letter}}
\address{\textsuperscript{1}Class of Science,
Scuola Normale Superiore, Pisa - Italy,\\
Department of Science, Technology and Society, University School for Advanced Studies IUSS Pavia, Pavia - Italy\\
Department of Mathematics,
Technische Universität Darmstadt, Darmstadt - Germany.}

\address{ \textsuperscript{2}Dipartimento di Matematica,
 Centro Interdipartimentale per lo Studio dei Sistemi Complessi Universitá di Pisa, Pisa -Italy}
 
\address{\textsuperscript{3, \Letter}School of Mathematics,
Monash University, Melbourne - Australia.}

\email{\textsuperscript{1} delsarto@mathematik.tu-darmstadt.de}
\email{\textsuperscript{2} stefano.galatolo@unipi.it}
\email{\textsuperscript{3, \Letter} sakshi.jain@monash.edu}

\date{\today}

\begin{document}
\maketitle 
\begin{abstract} 
  We consider a random dynamical system on $\mathbb{R}^d$, whose dynamics is defined by a stochastic differential equation. The annealed transfer operator associated with such systems is a kernel operator. Given a set of feasible infinitesimal perturbations $P$ to this kernel, with support in a certain compact set, and a specified observable function $\phi: \mathbb{R}^d \to \mathbb{R}$, we study which infinitesimal perturbation in $P$ produces the greatest change in expectation of $\phi$. We establish conditions under which the optimal perturbation uniquely exists and present a numerical method to approximate the optimal infinitesimal kernel perturbation. Finally, we numerically illustrate our findings with concrete examples.
\end{abstract}

\tableofcontents

\section{Introduction}\label{sec:intro}
 The predictive understanding and the control of the statistical properties  of a  dynamical system is an important topic of study, with applications to  many different fields, so is understanding the change in these statistical properties after small changes in the initial dynamical system.
The concept of statistical stability of a dynamical system relates to the statistical properties of typical orbits of a dynamical system, which are in turn encoded into its invariant or stationary measures, and to how these properties change when the system is perturbed. We say that the system exhibits a linear response to the perturbation when such a change is differentiable (see Theorem \ref{thm:res_diff} for a formalization of this concept).
In this case, the long-time average of a given observable  changes smoothly during the perturbation. 

 The linear response for dynamical systems was first studied by Ruelle \cite{Ruelle97}, and then was studied extensively for different classes of systems having some form of hyperbolicity (see \cite{Baladi14} for a general survey). The linear response for stochastic differential equations (SDEs) was studied in \cite{HM10}, \cite{KLP} and \cite{CKB22}, where general response results were established.
In the context of discrete-time random systems, linear response results have been shown in \cite{GG19}, \cite{BRS20}  and \cite{AFG22}. The study of linear response has great importance in the applications, in particular to climate sciences (see \cite{GL20}, \cite{HM10}).

In the present article, we address a natural inverse problem related to linear response: The Optimal Response of a given observable. We consider a certain observable, defined on the phase space associated with our system and search for the optimal infinitesimal perturbation to apply to the system in order to maximize the observable's expectation.
The understanding of this problem also has evident importance in the applications, as it is a formalization of the natural question of ``how to manage the system in such a way that its statistical properties change in a wanted direction'', hence an optimal control problem for the statistical properties of the system \cite{AFG22,BLL20, HM10,GL20, M18}. 

The optimal response problem for a fixed observable as described above was studied for the first time in \cite{ADF18}, for finite-state Markov chains.
Then the case of dynamical systems  whose transfer operators are kernel operators (including random dynamical systems that have additive noise) was considered in \cite{AFG22}.
The case of one-dimensional deterministic expanding circle maps with deterministic perturbations was studied in \cite{FG23}.
The above articles also consider the problem of optimizing the spectral gap and hence the speed of mixing. In \cite{G22}, an optimal response problem is considered where optimal coupling is studied in the context of mean-field coupled systems.
A problem strictly   related to the optimal response is the `linear request problem', which focuses on the search for a perturbation achieving a prescribed response direction \cite{BLL20,GP17,GG19,M18,K18}.
All of these studies are  in terms of understanding how a system can be modified to control the behavior of its statistical properties.\\

  Stochastic differential equations and random dynamical systems are widely used as models of climate and fluid dynamics related phenomena.  This strongly motivates the study of the optimal response for such systems.
  As mentioned above, the  optimal response  and the control of the statistical properties for discrete time  random dynamical systems  whose associated transfer operator is a kernel operator is studied in \cite{AFG22} and \cite{GG19}. In these papers, the phase space considered was compact.
   Since the transfer operator associated with a stochastic differential equation is a kernel operator (see Theorem \ref{MPZ} for a precise statement and estimates on the regularity of the kernel).
The results of these papers hence apply to the case of time discretizations of stochastic differential equations over compact spaces. The extension of these results to non-compact phase space is a non-trivial task, which requires the use of suitable functional spaces, in order to recover the ``compact immersion''-like  functional analytic properties, which are well known to be important to establish  spectral gap  for the associated transfer operator. 
Important stochastic differential models are formulated on noncompact spaces, such as $\mathbb{R}^d$. 
Thus, such an extension is strongly motivated.
An approach to the definition of suitable spaces for the transfer operators associated with stochastic differential equations on $\mathbb{R}^d$ was implemented in \cite{FGGV24}, by using a sort of weighted Bounded Variation spaces with the scope of studying extreme events.

In the present paper, motivated by these considerations, we approach the study of optimal linear response for random dynamical systems over $\mathbb{R}^d$ that have a kernel transfer operator. 
This setting is, in our opinion, a first step in approaching the study of the optimal response for  stochastic differential equations.

\section{Settings and Results}\label{sec:sde}

Let $T>0$ and consider the SDE on $\mathbb{R}^d$ given by
\begin{equation}  \label{eq:system1}
\left \lbrace 
\begin{aligned}
dX_{s}^x & =b\left(X_{s}^x\right) dt+dW_{s}, &\quad s \in (0,T),\\ X_{0}^x &
=x.
\end{aligned}
\right.
\end{equation}
Here, $(W_{s})_s$ is a Brownian motion defined on a probability space $\left(
\Omega,\mathcal{F},\mathbb{P}\right)$, and $x \in \mathbb{R}^d$ is the initial condition. We impose the following assumptions on the drift $b$: 

\begin{description}
\item[A] (Locally-Lipschitz continuity) For each \(x_0\in\bR^d\), there exist constants $K>0$ and \(\delta_0>0\) such that for all $ x \in {\mathbb{R}^d}$ with \(|x_0-x|<\delta_0\),
\begin{equation}\label{eq:lipschitz}
|b(x_0)-b(x)|\leq K|x_0-x|.
\end{equation}

\item[B] (Dissipativity) There exist constants $c_1 \in \mathbb{R}$ and $c_2>0$ such that for all $x \in \mathbb{R}^d $, 
\begin{equation}  \label{eq:dissiphyp}
\langle b(x), x \rangle \leq c_1 - c_2 \norm{x}^2 .
\end{equation}
\end{description}

Under assumptions \textbf{A} and \textbf{B}, the SDE \eqref{eq:system1} admits a unique stationary measure $\mu$, see \cite[Section~4.1]{FGGV24} for details.

Now, let $\theta _{s}$, for $s\geq 0$, be the flow associated to the deterministic part of the above SDE, that is
\begin{equation*}
\left\{ 
\begin{array}{l}
\dot{\theta}_{s}(x)=b(\theta_{s}(x)), \\ 
\theta _{0}(x)=x.%
\end{array}%
\right.
\end{equation*}
For $\lambda \in (0,1],\  s>0$, define the Gaussian density
\begin{equation*}
g_{\lambda}(s,x):=s^{-\frac{d}{2}}e^{\frac{-\lambda |x|^{2}}{s}}.
\end{equation*}
The following result, due to \cite{MPZ}, provides two-sided estimates and gradient bounds for the density associated to the (unique) solution $(X_s^x)_s$ of the SDE \eqref{eq:system1}.

\begin{theorem}[\protect\cite{MPZ}, Theorem 1.2 and Remark 1.3.]
\label{MPZ} 
Fix \(T>0\).  For each \(t\in(0,T)\) and \(x\in\mathbb{R}^{d}\), the law of \(X_{t}^{x}\) has a density $\kappa_t(x,y)$ that is
continuous in both variables $x,y \in \mathbb{R}^d$. Moreover, $\kappa_t$ satisfies the following:

\begin{description}
\item[1] (Two sided density bounds) There exist  constants $\lambda _{0}\in
(0,1], ~C_{0}\geq 1$, depending on $T,K,d$, such that for all $x,y\in \mathbb{R}%
^{d}$ and $~t<T$,%
\begin{equation*}
C_{0}^{-1}g_{\lambda _{0}^{-1}}(t,\theta _{t}(x)-y)\leq \kappa_t(x,y)\leq
C_{0}g_{\lambda _{0}}(t,\theta _{t}(x)-y).
\end{equation*}

\item[2] (Gradient estimates) There exist constants $\lambda _{1}\in (0,1],~C_{1}\geq 1,
$ depending on $T,K,d$, such that for all $x,y\in \mathbb{R}^{d}$, and $~t<T$%
\begin{equation*}
|\nabla _{x} \kappa_{t}(x,y)|\leq C_{1}t^{-\frac{1}{2}}g_{\lambda _{1}}(t,\theta
_{t}(x)-y),
\end{equation*}%
\begin{equation*}
|\nabla _{y} \kappa_{t}(x,y)|\leq C_{1}t^{-\frac{1}{2}}g_{\lambda _{1}}(t,\theta
_{t}(x)-y).
\end{equation*}
\end{description}
\end{theorem}

\begin{remark}\label{rem:no_t}
    As stated above, the SDE \eqref{eq:system1} that we consider in this article admits a unique stationary measure $\mu$.
   In the following, we fix the time $t=1$ and consider a time discretisation of this system. For notational convenience, we drop the subscript $t$, and denote the kernel $\kappa_1(x,y)$ simply by $\kappa(x,y)$.
\end{remark}

As discussed in Section~\ref{sec:intro}, a dynamical system exhibits \emph{linear response} if its invariant measure changes differentiably with respect to the changes in the initial system. One of our main results gives the \emph{explicit characterization for linear response} for the system \eqref{eq:system1}, which is suitable for the numerical approximation. 

To be able to state the precise result (Theorem~\ref{thm:res_diff}), we introduce the transfer operator \(\cL: L^1 \to L^1\) (Definition~\ref{def:to}), defined by
$$
\cL f(y) = \int \kappa (x,y) f(x)\ dx,
$$
where  $\kappa$ is the kernel from Theorem \ref{MPZ} (see Remark \ref{rem:no_t}). 

To get a spectral gap and other desirable properties, we consider $\mathcal{L}$ as acting on  a suitable Banach space \(\fB\) of weighted densities (see \eqref{def:strongsp}).
 We then consider a class of perturbed systems by applying perturbations directly to the associated transfer operators. We consider consider a compact domain $D\subset \mathbb{R}^d$ and we  introduce  a family of \emph{perturbed kernels} $\kappa_\delta$ of the form
$$
\kappa_\delta = \kappa + \delta \cdot  \dot   \kappa + r_\delta,
$$
where $\dot \kappa, r_\delta \in L^2(D \times D)$ and $r_\delta = o (\delta)$ is an higher-order perturbation in $L^2(D \times D).$ This yields a family of perturbed transfer operators
 \begin{equation}\label{pert}
\cL_\delta f(y) = \int \kappa_{\delta }(x,y)  f(x)\ dx,
\end{equation}
with the convention \(\cL_0=\cL\), see Section~\ref{sec:perturb} for more details.

\begin{maintheorem}

    Consider the family \(\cL_\delta:\fB\to\fB\) of transfer operators, with \(\delta\in [0,\bar\delta)\). Then:
    \begin{enumerate}
        \item The operators have invariant densities in \(\fB\): for each \(\delta\in [0,\bar\delta)\) there is \(f_\delta\neq 0\) such that \(\cL_\delta f_\delta=f_\delta\) (Proposition~\ref{prop:inv_den} and Proposition~\ref{prop:inv_msr}).
        \item There exists an operator \(\dot\cL:L^1\to L^1\), defined by \(\dot\cL f(y) = \int \dot\kappa(x,y) f(x)\ dx\), such that, restricted to the strong space \(\fB\), the following limit holds
        \[
        \lim\limits_{\delta\to 0} \norm{\frac{\cL_\delta-\cL_0}{\delta}f_0-\dot\cL f_0}_s=0,
        \]
        where \(\norm{\cdot}_s\) is the norm on the strong space \(\fB\) (Lemma~\ref{ldot2}).
        \item For \(\delta\) small enough, we have linear response: that is 
        \begin{equation}\label{res0}
        \lim\limits_{\delta\to 0}\norm{\frac{f_\delta-f_0}{\delta}-(Id-\cL_0)^{-1}\dot\cL f_0}_1=0,
        \end{equation}
        where \(\norm{\cdot}_1\) is the \(L^1\) norm (Theorem~\ref{thm:res_diff}).
    \end{enumerate}
    
\end{maintheorem}

While results on the existence of linear response for SDEs on $\mathbb{R}^d$ can be found in the literature (see \cite{HM10}, \cite{CKB22}), the formula \eqref{res0} is particularly suitable for numerical approximation, within an appropriate numerical scheme that we describe in Section~\ref{sec:optimal}. We exploit this formulation to pursue the main objective of the paper: identifying the optimal perturbation that induces a prescribed response, as outlined in the introduction. Thanks to the explicit characterization that we derive, we are also able to demonstrate how to approximate such perturbations in concrete examples (see Section~\ref{sec:numerics}).

This leads to our next main result, which addresses the existence and uniqueness of such an optimal perturbation. 
To state it, consider a compact neighborhood \(D\subset  \bR^d\) of $0$. Formalizing the idea that we mean to perturb the kernel of the transfer operator in a certain direction $\dot \kappa \in L^2(D\times D)$, we introduce an operator (see Definition~\ref{R}) 
$$
R: L^2(D\times D) \to L^1, \quad R(\dot \kappa) = \lim\limits_{\delta\to 0}\frac{f_{\dot \kappa,\delta} - f_0}{\delta},
$$
where $f_{\dot \kappa,\delta}$ and $ f_0$ are the invariant densities of the perturbed and unperturbed transfer operators \(\cL_\delta\) and \(\cL_0\), respectively, acting on the Banach space \(\fB\). 

\begin{maintheorem}
Let $\phi \in L^\infty $ be an observable, and let $P \subset L^2(D \times D)$ be a closed, bounded, and strictly convex set whose relative interior contains the zero function. Then the optimization problem (Proposition~\ref{prop:gen_cnvx})
     \begin{eqnarray} 
     \max \left \lbrace \mathcal{J}(\dot{\kappa})\ : \   \dot{\kappa} \in P\right \rbrace 
\end{eqnarray}
admits a unique solution.
\end{maintheorem}

The above results are stated more precisely and proved in the following sections (see the references provided in each statement), where all the necessary definitions and technical details are also introduced.

Finally, in subsection~\ref{sec: numerical scheme for the approximation of the opt. pertb.}, we present a numerical scheme for approximating the optimal perturbation. This scheme is built upon the constructive nature of our proofs and the explicit formulas derived therein, which are directly implemented in the numerical procedure. A constructive algorithm for computing the optimal perturbation is given. The corresponding numerical experiments are carried out on a concrete example of an SDE in Section~\ref{sec:numerics}.

 We remark that because of the type of perturbations to the system considered in this work (see \eqref{pert}) the transfer operator obtained after perturbation may no longer be associated with an SDE. Such a perturbation can be interpreted as a ``local'' change to the initial system (the one arising from an SDE), whose nature is independent of the phenomenon whose model is the SDE. This can  be seen as a first step in the study of the optimal response for SDEs, where we apply the simplest possible meaningful perturbations. However one would like to consider other ``non-compact'' perturbations as well, such as those arising from perturbations to the drift term $b$ of the SDE (see \eqref{eq:system1}).

A natural direction for future work is indeed to consider perturbations of the system, to be applied directly on the SDE defining it, and in particular to the drift term $b$.
We think that this case also would fit our general framework. 
However in order to achieve this,
the differentiability and the differential of the transfer operator with respect to this perturbation should be obtained, like it is done
in Section \ref{sec:perturb} for the kind of perturbations we consider. Similar estimates have been achieved in \cite{KLP}. Unfortunately the convergence of these estimates in \cite{KLP} is in the $L^2$ topology, which is not sufficient for our purposes
(see Remark \ref{R20} for an explanation on why $L^2$ is not a good space in order to consider to get spectral gap for the transfer operator, and hence apply our strategy for the optimal response when the phase space is  $\mathbb{R}^d$ and hence not compact).
However such result indicates that with some further work the goal of fitting this kind of perturbations in our framework for the optimal response is achievable.

\subsection*{Overview of the paper}

The remainder of this paper is organised as follows. In Section~\ref{sec:to}, we introduce the transfer operator associated with the SDE~\eqref{eq:system1}, see~\eqref{eq:stoctransfer}. In particular, it is a kernel operator on the non-compact phase space $\mathbb{R}^d$. Then, we define suitable spaces for this operator. Specifically, we let the transfer operator act on $L^1(\mathbb{R}^d)$ and on a ``stronger'' space, constructed using a space we denote by 
$\fB$, which consists of $L^1(\mathbb{R}^d)$ densities that decay at infinity with a prescribed speed and are also in $L^2 $ on the domain $D$ where the perturbations are applied. The choice of the $L^2$ topology in this domain is motivated by optimisation purposes, which are simplified when working on a Hilbert space. The prescribed decay at infinity will be useful in obtaining suitable compactness properties, which will imply that the transfer operator, when considered on the strong space, has a spectral gap (see Lemma~\ref{lem:simple_ev}) and strong mixing properties (see Proposition~\ref{prop:bdd_res}).

In Section~\ref{sec:perturb}, we define and study the perturbations we intend to apply to our system. We show that the perturbed operators also have a spectral gap, and we prove a linear response statement for these systems and perturbations (see Theorem~\ref{thm:res_diff}).

In Section~\ref{sec:optimal}, we consider an observable in $L^\infty(\mathbb{R}^d, \mathbb{R})$ and explore the problem of finding the optimal perturbation that maximises the rate of change of the expectation of the observable. We formalise the optimisation problem in Problem~\ref{prob:max}, and prove that it has a unique solution (see Proposition~\ref{prop:gen_cnvx}). Moreover, we describe a numerical approach to approximate the unique solution in the case where the set of feasible perturbations is a ball in a suitable Hilbert space.

In Section~\ref{sec:numerics}, we apply the algorithm to illustrate the optimal perturbation on some examples. In particular, we consider a gradient system SDE with a symmetric double-well potential. Then, via a finite difference method, we numerically approximate the solution of the optimisation problem for a smooth observable given by the probability density function of a Gaussian random variable.

Finally, there is an appendix stating some well known results of convex optimisation that we have extensively used in Section~5.

\section{Transfer operator and Banach spaces}\label{sec:to}
\subsection{The Kolmogorov operator and the transfer operator}
In this section, we define the transfer operator associated with the evolution of the SDE considered at time $t=1$ and state/prove some basic properties of these operators. We will extensively use the density \(\kappa(x,y)\) (as Remark~\ref{rem:no_t} says, we drop the subscript \(t\)) and its properties given by Theorem~\ref{MPZ}.
\begin{definition}
\label{def:koldef} The Kolmogorov operator $\cP:L^{\infty }(\mathbb{R}%
^{d})\rightarrow C^{0}(\mathbb{R}^{d})$ associated with the system %
\eqref{eq:system1} at time $t=1$ is defined as follows. Let $\phi \in L^{\infty }(%
\mathbb{R}^{d})$, then for all $ x\in \mathbb{R}^{d}$ we set
\begin{equation*}
(\cP\phi )(x):=\mathbb{E}[\phi (X_1^x)],
\end{equation*}
where $X_1^x$ is the solution at time $t =1$ of the SDE \eqref{eq:system1} with initial condition $x$.
\end{definition}

The solution $X_1^x$, thanks to Theorem~\ref{MPZ}, has a density \ $\kappa(x,y)$, and thus we have
 
\begin{equation*}
(\cP\phi )(x)=\int_{\mathbb{R}^{d}}\phi (y)\kappa(x,y)dy.
\end{equation*}%
By this we see that 
\begin{equation}
\norm{ \cP\phi }_{\infty }\leq \norm{ \phi}_{\infty }.  \label{inft}
\end{equation}

 If $\nu $ is a Borel signed measure on $%
\mathbb{R}^{d}$ 
\begin{equation*}
\int_{\mathbb{R}^{d}}(\cP\phi )(x) d\nu(x) = \int_{\mathbb{R}^{d}}\int_{\mathbb{R%
}^{d}}\phi (y)\kappa(x,y)dyd\nu (x)
\end{equation*}%
supposing that $\nu $ has a density with respect to the Lebesgue measure {%
 $f\in L^{1}(\mathbb{R}^{d})$ i.e. $d\nu =f(x)dx$. We can thus
write 
\begin{equation}
\int_{\mathbb{R}^{d}}(\cP\phi )(x)d\nu (x)=\int_{\mathbb{R}^{d}}\phi (y)\left(
\int_{\mathbb{R}^{d}}\kappa(x,y)f(x)dx\right) dy.  \label{duality}
\end{equation}%
}
Now we define the transfer operator $\cL:L^{1}(\mathbb{R}^{d})\rightarrow
L^{1}(\mathbb{R}^{d})$ associated with the evolution of the system at time $t=1$.

\begin{definition}[Transfer operator]\label{def:to}
Given $f\in L^{1}(\mathbb{R}^{d})$ we define the measurable function $\cL f:%
\mathbb{R}^{d}\rightarrow \mathbb{R}^{d}$ as follows. For almost each $y\in 
\mathbb{R}^{d}$ let 
\begin{equation}
\lbrack \cL f](y):=\int \kappa(x,y)f(x)dx.  \label{eq:stoctransfer}
\end{equation}
\end{definition}

\begin{remark}
We will show that this operator has a unique invariant density \( f_\mu \), satisfying \( f_\mu = \mathcal{L} f_\mu \). This function \( f_\mu \) is the density of the stationary measure \( \mu \) associated with the SDE~\eqref{eq:system1} (see Proposition~\ref{prop:inv_den}).

\end{remark}

By \eqref{duality} we now get the duality relation between the Kolmogorov
and the transfer operator 
\begin{equation}
\int (P\phi )(x)d\nu (x)=\int \phi (y)[\cL f](y)dy.  \label{eq:stocduality}
\end{equation}
The following are some well-known and basic facts about integral operators with
kernels $\kappa$ in $L^{p}$, which will be useful:\\

\begin{itemize}
\item If $\kappa\in L^2$ the operator $\cL:L^{2}\rightarrow L^{2}$ is bounded and  
\begin{equation}
\norm{\cL f}_{2}\leq \norm{\kappa}_{2}\norm{f}_{2}  \label{KF}
\end{equation}%
(see Proposition 4.7 in II.\S 4 \cite{C}).

\item If $\kappa\in L^{\infty}$, then 
\begin{equation}
\norm{\cL f}_{\infty }\leq \norm{\kappa}_{\infty }\norm{f}_{1}  \label{KF2}
\end{equation}
and the operator $\cL:L^1\rightarrow L^{\infty }$ is bounded.
\end{itemize}
Now, we discuss the properties of the transfer operator on the space \(L^1\).
\begin{lemma}\label{lem:contraction}
The operator $\cL$ preserves the integral and is a weak contraction with
respect to the $L^1$ norm.\label{7}
\end{lemma}

\begin{proof}
The first statement directly follows from \eqref{eq:stocduality} setting $%
\phi =1$. For the second, we can work similarly using \eqref{inft} and %
\eqref{eq:stocduality} with $\phi =\frac{[\cL f]}{|[\cL f]|}$.
\end{proof}

\begin{remark}\label{rem:Markov_op}
    Since $\cL$ is a positive operator, we also get that $\cL$ is a Markov operator having kernel $\kappa$.
\end{remark}

Finally, in the following section, we define/construct the strong space that we want to study the operator \(\cL\) on.
\subsection{Functional Spaces}\label{sec:spaces}

In this section we construct spaces which are  suitable to obtain spectral gap on non-compact domains for the action of the transfer operator associated to the system. 

 The spaces we consider consist of densities whose behavior far away from the origin is controlled using certain weight functions going to $\infty $ at infinity, following the approach of \cite{FGGV24}. 
 
Let $\alpha > 0 $ and define the weight function 
\begin{equation}  \label{WF}
\rho_{\alpha} \left( \left\vert x\right\vert \right) =\left( 1+\left\vert
x\right\vert ^{2}\right) ^{\alpha/2}.
\end{equation}
Let $L_{\alpha}^{1}\left( \mathbb{R}^{d}\right) $ be the space of Lebesgue
measurable $f:\mathbb{R}^{d}\rightarrow\mathbb{R}$ such that 
\begin{equation*}
\left\Vert f\right\Vert _{L_{\alpha}^{1}\left( \mathbb{R}^{d}\right)
}:=\int_{\mathbb{R}^{d}}\rho_{\alpha} \left( \left\vert x\right\vert \right)
\left\vert f\left( x\right) \right\vert dx<\infty. 
\end{equation*}
Note that, $L^1_{\alpha} \subset L^{1} $ and for $f \in L^1_{\alpha} $, $\norm{ f }_1 \leq \norm{ f }_{L^1_{\alpha}} $. Moreover for $\alpha= 0$, $L^1_0 = L^1$. 

\begin{remark}
    Throughout the article, we denote the \(L^1, L^2\) and \(L^\infty\) norms by \(\norm{\cdot}_1, \norm{\cdot}_2\) and \(\norm{\cdot}_\infty\) respectively.
\end{remark}

For a Borel subset $S\subseteq \mathbb{R}^{d}$ let us define 
 \begin{equation*}
 \operatorname{osc}(f,S)= \esssup_{x\in S}f -\essinf_{x\in S}f.
 \end{equation*}
 
 Let  $\psi$ be a Radon probability measure on $\mathbb{R}^d$ and assume that: 

\begin{itemize}
\item[\textbf{($A_{\protect\psi}1$)}] $\psi$ is absolutely continuous with respect to
the Lebesgue measure, having a continuous bounded density $\psi^{\prime }$
such that $\psi^{\prime }>0$ everywhere. 
\end{itemize}

With such a  probability measure satisfying (\(A_\psi 1\)), define a norm $\norm{ \cdot }_{BV_{\alpha}}$ for \(f\in L^1_\alpha(\bR^d)\)  by setting 
\begin{equation}  \label{def:BValpha}
\norm{ f}_{BV_{\alpha}}:=\norm{f}_{L_{\alpha
}^{1}\left( \mathbb{R}^{d}\right) }+\sup_{\epsilon\in\left( 0,1\right]
}\epsilon^{-1}\int_{\mathbb{R}^{d}}\operatorname{osc}\left( f,B_{\epsilon}\left( x\right)
\right) d\psi(x),
\end{equation}
where $B_{\epsilon }(x)$ denotes the ball in $\mathbb{R}^d$ with center $x$ and radius $\epsilon$.
Being the sum of a norm and a seminorm, $\norm{ \cdot }_{BV_{\alpha}}$ indeed defines a norm and the following space is a Banach space\footnote{In \cite{S2000} and \cite{K85}, these spaces have been studied extensively and in more general form as well.} (see \cite[Proposition~3.3]{S2000})
\[
BV_\alpha(\bR^d)=\{f\in L^1_\alpha: \norm{f}_{BV_\alpha}<\infty\}.
\] 
 The presence of the measure $\psi$ in the definition of the oscillatory seminorm is necessary to treat the case when we have a non-compact domain, which in this case is $\mathbb{R}^d$. Instead, if the domain was compact, one could simply take $\psi=1$.\\

Finally, to define the strong space, let $D$ be a compact set in $\mathbb{R}^{d}$. In what follows, \(D\) is the set where we allow perturbations in our system. Accordingly, we define a
suitable norm, adapted to such perturbations, as follows:
\begin{equation}\label{eq:strong_norm}
\norm{f}_{s}:=\norm{f}_{L_{2}^{1}}+\norm{1_{D}f}_{{2}}.
\end{equation}

We define the strong space of densities as follows, equipped with the norm $\norm{\cdot}_{s}$:
\begin{equation}\label{def:strongsp}
\fB:=\{f\in L_{2}^{1}\left( \bR^{d}\right) ,\norm{f}_{s}<\infty \}.
\end{equation}

The following results from \cite{FGGV24} will be useful to prove the compact inclusion of the strong space \(\fB\) in \(L^1\) and, ultimately, the spectral gap.
\begin{proposition}
\label{thm:embeddingRD} $BV_{\alpha }\left( \mathbb{R}^{d}\right)
\hookrightarrow L^{1}\left( \mathbb{R}^{d}\right) $ is a compact embedding.
\end{proposition}

\begin{lemma}
\label{lem:pbound} There exist constants $A,B>0$ and $\lambda \in \left(
0,1\right) $  such that 
\begin{equation}
\norm{\cL^{n}f}_{L_{2}^{1}}\leq A\lambda ^{n}\norm{f}_{L_{2}^{1}}+B\norm{f}_{{1}}
\label{ineq1 lemma 18}
\end{equation}%
for every $f\in L_{2}^{1}$ and every $n\in \mathbb{N}$.
\end{lemma}

\begin{lemma}
\label{lem:Rdregularization} For every $t>0$, $\cL$ is bounded linear from $%
L^{1}$ to $C^{1}$; in particular, there exists $C>0$ such that 
\begin{equation}
\norm{\cL f}_{C^{1}}\leq C\norm{f} _{{1}}.
\label{C1}
\end{equation}%
Moreover if $f\in L_{2}^{1}$ then 
\begin{equation}
\norm{ \cL f}_{BV_{2}}\leq C\norm{ f}_{L_{2}^{1}}.  \label{eq:RdstrongLY}
\end{equation}
\end{lemma}
The proofs of Proposition~\ref{thm:embeddingRD}, Lemma~\ref{lem:pbound} and Lemma~\ref{lem:Rdregularization} can be found in \cite[Theorem~16, Lemma~19 and Lemma~20]{FGGV24} respectively.\\

Now we proceed to prove that the transfer operator  $\cL$ has spectral gap when acting on $\fB$, for which we need some preliminary results like compact inclusion (Proposition~\ref{prop:cpt_incl}) and Lasota-Yorke inequality (Lemma~\ref{lem:LY-ineq}) etc.\\
By Proposition \ref{thm:embeddingRD}, and Lemma \ref{lem:Rdregularization} we have the following:

\begin{proposition}\label{prop:cpt_incl}
Let $B$ be the closed unit ball in $\fB$. Then $\cL(B)$ is a compact set in $L^{1}$.
\end{proposition}

\begin{proof} Using Proposition~\ref{thm:embeddingRD}, which states that \(BV_2\) is compactly embedded in \(L^1\), it suffices to prove that \(\cL(B)\) is a closed and bounded set in the space \(BV_2\). 
By Lemma~\ref{lem:Rdregularization}, for every \(f\in B\), there exists \(C>0\) such that \(\norm{\cL f}_{BV_2}\leq C\norm{f}_{L^1_2}\leq C\norm{f}_s\leq C\), which implies \(\cL(B)\) is a bounded set in \(BV_2\). \\
To prove that \(\cL(B)\) is closed, let \((x_n)_{n\in\bN}\) be a Cauchy sequence in \(\cL(B)\). Since \(\cL(B)\subset BV_2\) and \(BV_2\) is a Banach space, there exists \(x\in BV_2\subset L^1\) such that \(\norm{x_n- x}_{BV_2}\to 0\). We want to show that \(x\in \cL(B)\). Since \(x_n\in \cL(B)\), there exists a sequence \((y_n)_{n\in\bN}\in B\) such that, for all \(n\in\bN\), \(\norm{y_n}_s\leq 1\) and \(\cL(y_n)=x_n\). As the unit ball \(B\) is closed, there exists \(y\in B\) such that \(\norm{y_n-y}_s\to 0\). By continuity, we have \(\cL(y)=x\), which implies \(x\in \cL(B)\), completing the proof.
\end{proof}

\begin{lemma}\label{lem:LY-ineq}
There exist constants $A,B>0$ and $\lambda \in \left( 0,1\right) $ such that 
\begin{equation}
\norm{\cL^{n}f}_{s}\leq A\lambda ^{n}\norm{f}_{s}+B\norm{ f}_{L^{1}}  \label{SLY}
\end{equation}%
for every $f\in \fB$ and every $n\in \mathbb{N}$.
\end{lemma}

\begin{proof}
By Lemma \ref{lem:pbound}, \ $\exists\  A_{1},B_{1}\geq 0,\lambda \in (0,1)$
such that%
\begin{eqnarray*}
\norm{\cL^{n}f}_{s}
&=&\norm{\cL^{n}f}_{L_{2}^{1}}+\norm{1_{D}\cL^{n}f}_{{2}} \\
&\leq &A_{1}\lambda ^{n}\norm{f} _{L_{2}^{1}}+B_{1}\norm{
f}_{{1}}+\norm{1_{D}\cL^{n}f}_{{2}}.
\end{eqnarray*}

Furthermore, by \eqref{KF2} and Lemma~\ref{lem:contraction}, 
\begin{eqnarray*}
\norm{1_{D}\cL^{n}f}_{{2}} &\leq &\sqrt{m(D)}\norm{\cL^{n}f}_{\infty } \\
&\leq &\sqrt{m(D)}\norm{\kappa}_{\infty }\norm{\cL^{n-1}f}_{{1}} \\
&\leq &\sqrt{m(D)}\norm{\kappa}_{\infty }\norm{f}_{{1}}
\end{eqnarray*}%
which leads to $\eqref{SLY}$.\footnote{\(m(D)\) denotes the $d$-dimensional Lebesgue measure of \(D\). }
\end{proof}
Let us define the spaces of zero-average functions in \(\fB\) and \(L^1\) respectively as 
\begin{equation} \label{V_B}
    V_{\fB}=\{f\in \fB:~\int fdm=0\}
\end{equation}
and 
\begin{equation}\label{V_s}
    V_{L^1}=\{f\in L^1:~\int fdm=0\}.
\end{equation}

In the rest of this section, we prove results on the functional analytic properties of the operator \(\cL\), namely, the existence of eigenvalues on the unit circle, the spectral gap, and the boundedness of the resolvent. These results, along with other inferences, give us the uniqueness of the invariant measure of the system in the strong space and its convergence to equilibrium. To prove these results, we extensively use the positivity of the kernel \(k\) associated with the operator \(\cL\). 

\begin{lemma}\label{lem:simple_ev}
The transfer operator \(\cL\) has spectral gap on $\fB$ and it has 1 as its simple and only eigenvalue on the unit circle.
\end{lemma}
\begin{proof}

    By Lemma~\ref{lem:contraction}, the operator \(\cL\) is integral preserving, which implies that \(1\) lies in the spectrum of \(\cL\).\\
   Recall that, by Lemma~\ref{lem:contraction} and Lemma~\ref{lem:LY-ineq}, the operator \(\cL\) satisfies the hypothesis of Hennion theorem (see \cite[Theorem~B.14]{DKL21}), which implies that the operator is quasi-compact which, in turn, implies that it has spectral gap on \(\fB\). Consequently, 1 is an eigenvalue of \(\cL\). \\
   Now, we want to prove that 1 is the unique eigenvalue of \(\cL\) on the unit circle. Let us suppose, to the contrary, that there exists an eigenvalue \(\theta\neq 1\) such that \(|\theta|=1\). Accordingly, there exists a corresponding eigenvector \(f_\theta\in\fB\), that is 
\[
\cL f_\theta=\theta f_\theta.
\] 
Then, for any \(n\in\bN\), 
\[
\begin{split}
    \norm{\cL^{n+1}f_\theta-\cL^n f_\theta}_{1}&=\norm{\theta^n\cL f_\theta-\theta^n f_\theta}_{1}\\
    & = \norm{\cL f_\theta-f_\theta}_{1}.
\end{split}
\]
Since \(\cL\) preserves the integral (Lemma~\ref{lem:contraction}), we have 
\[
\int (\cL f_\theta-f_\theta) = \int \cL f_\theta-\int f_\theta=\int f_\theta-\int f_\theta =0.
\]
 This implies \(\cL f_\theta-f_\theta(\neq 0)\in V_\fB\). Then, since the kernel \(k\) is positive (by Theorem~\ref{MPZ}), we get
\[
\begin{split}
    \norm{\cL^n(\cL f_\theta-f_\theta)}_{1}&<\norm{\cL^{n-1}(\cL f_\theta-f_\theta)}_{1}\\
&<\norm{\cL f_\theta-f_\theta}_{1}.
\end{split}
\]
This contradicts the earlier calculation. Hence, 1 is the only eigenvalue of \(\cL\) on the unit circle. \\
Now, to show that 1 is a simple eigenvalue, let \(u,v\in\fB\) be two distinct eigenvectors corresponding to the eigenvalue 1. Then 
\[\cL u=u,\  \cL v=v .\]
Since \(u\) and \( v\) are fixed points of \(\cL\), \(u-v\) is also a fixed point of \(\cL\) and hence an invariant measure. Notice that \(\int u-v=\int u-\int v=0\) which implies \(u-v\in V_\fB\). Again, since \(k\) is positive, we have \(\norm{\cL(u-v)}_s<\norm{u-v}_s\), which contradicts the fact that \(u-v\) is a fixed point. Consequently, 1 is a simple eigenvalue of \(\cL\) on \(\fB\).
\end{proof}

A first direct consequence of Lemma \ref{lem:simple_ev} is the uniqueness of the invariant probability measure of $\cL$.

\begin{proposition}\label{prop:inv_den}
    The operator has a unique  invariant probability density \(f\in\fB\). Hence satisfying \(\int f=1\) and \(\cL f=f\).
\end{proposition}

\begin{remark}
   We remark that the unique stationary measure $\mu$ of the SDE \eqref{eq:system1}
     must be also invariant for $\cL$, and this must be the unique invariant probability measure found above. This also implies that $\mu$ has a density in the strong space $\fB$ (see Remark~\ref{rem:no_t}). 
   We conclude that the unique stationary measure of the SDE is indeed equal to the invariant measure of the transfer operator.
\end{remark}

The spectral gap and the uniqueness of the invariant probability measure for $\cL$ imply the following two classical consequences, which will be used later.

\begin{proposition}\label{prop:bdd_res}
 For each $f\in V_\fB$ 
\begin{equation*}
\lim_{n\rightarrow +\infty }\norm{\cL^{n}(f)}_{{s}}=0,
\end{equation*}
and the resolvent operator $(Id-\cL)^{-1}:V_{\fB}\rightarrow V_{\fB}$ is a bounded operator.
\end{proposition}

\begin{proof}
  By Lemma~\ref{lem:simple_ev} and the fact that the operator is positive and integral preserving, we get that \(\cL:\fB\to\fB\) has spectral gap, which implies it can be decomposed as
\[
\cL=\Pi+{\mathcal Q}
\]
where \(\Pi\) is the projection operator onto the one-dimensional eigenspace corresponding to the eigenvalue 1, and the spectral radius of \({\mathcal Q}\) is strictly less than 1, that is \(\norm{{\mathcal Q}}_{s}<1\). \\ 
Furthermore, for \(f\in V_\fB\) we have \(\cL^n f={\mathcal Q}^n f\), and
\[
\norm{{\mathcal Q}^n f}_{s}\leq \norm{{\mathcal Q}^n}_{s}\norm{f}_{s} \leq Ce^{-\lambda n}\norm{f}_{s}
\]
for some $C\geq0 $ and $\lambda \in (0,1)$. The first claim follows directly by this estimate. The second claim also follows from the estimate, since the resolvent operator \((Id-\cL)^{-1}:V_\fB\to V_\fB\) can be written as \((Id-\cL)^{-1}=\sum_{n=0}^\infty \cL|_{V_\fB}^n\).
\end{proof}

\begin{remark}\label{R20}
We have seen that the transfer operator is quasicompact when acting on $L^1_2$. 
In this remark we discuss the optimality of the result  and the importance of $L^1_2$ to obtain such a result, showing that the transfer operator neither always have a spectral gap on $L^2$, nor is always a compact operator on $L^1_2$ and hence we cannot rely on a simple spectral perturbation theory for such operators.

For simplicity let us illustrate this in the particularly simple, but meaningful case where $d=1$ and  $b(x)=-x$.

We hence have the following SDE on the real line:
\[
dX_t = -X_t \, dt + dW_t.
\]
We will exploit the fact that for such an SDE there are explicit formulas for the evolution of densities through the transfer operator (see e.g. \cite{gardiner2004handbook}) when the initial condition is distributed as a Gaussian.

It is well known in fact that if for such an SDE the initial condition is
\[
X_0 \sim \mathcal{N}(\mu_0, \sigma_0^2),
\]
then \( X_t \) has Gaussian distribution for all \( t \geq 0 \), i.e.
\[
X_t \sim \mathcal{N} \left( \mu(t), \sigma^2(t) \right),
\]
where the average is
$$
    \mu(t) = \mathbb{E}[X_t] = \mu_0 e^{-t},
$$
and the variance is
$$
    \sigma^2(t) = \mathrm{Var}(X_t) = \sigma_0^2 e^{-2t} + \frac{1}{2} \left( 1 - e^{-2t} \right).
$$

Let $V_{L^2}=\{f\in L^2:~\int fdm=0\}$.
By a reasoning similar to the one done in the proof of Lemma \ref{lem:simple_ev} and Proposition  \ref{prop:bdd_res}, if the transfer operator $\mathcal{L}$ had a spectral gap on $L^2$, there was an iterate $\mathcal{L}^{n_0}$ with $n_0$ big enough such that 
$||\mathcal{L}^{n_0}|_{V_{L^2}}||_2\leq
\frac{1}{2}$.
Let $x_n \in \mathbb{R} $ be a sequence of points of the type $x_n= 2e^n$, and let us consider a sequence of $L^2$ distributions in the unit sphere $f_n= \frac{\mathcal{N}(x_n, 1)-\mathcal{N}(-x_n, 1)}{||\mathcal{N}(x_n, 1)-\mathcal{N}(-x_n, 1)||_2}$.  Now, $n_0$ big enough corresponds to $t$ large enough in the Fokker-Planck equation, leading 
to the fact that $L^{n_0}(f_n)$ is the difference of two Gaussian distributions with variance which is smaller than $1$  while the two averages diverge as $n \to \infty $.  
We have then that $\lim_{n\to \infty}\frac{||L^{n_0}(f_n)||_2}{||f_n||_2}\geq 1$, contradicting spectral gap.

Similarly we can remark that   $\mathcal{L}$ is not in general a compact operator when acting  on $L^1_2$ by noting that the sequence  $g_n=\frac{\mathcal{N}(x_n, 1)}{||\mathcal{N}(x_n, 1)||_{L^1_2}}$ is a bounded sequence  such that $\mathcal{L} g_n$ has no converging subsequences.
\end{remark}

\section{Perturbations and linear response}\label{sec:perturb}
From now on, we denote the original transfer operator \(\cL\) defined above by \(\cL_0\).\\
As mentioned in Section~\ref{sec:intro}, to perturb the initial stochastic differential equation \eqref{eq:system1} that we began with, we perturb the associated transfer operator \(\cL_0\) (defined in Section~\ref{sec:to}). Since the transfer operator $\cL_0$ associated with a time discretisation $t$ of the system is a kernel operator (see Theorem \ref{MPZ} and \eqref{eq:stoctransfer}), we perturb the operator by perturbing the associated kernel. We would like to emphasise the fact that this method of perturbation has the limitation that the perturbed system may not correspond to a SDE. Rather, it signifies a ``local'' change in the initial SDE, whose nature is independent of the SDE model. 

The local kernel perturbations are conducted in the following way:\\
let \(D\subset \bR^d\) be a compact set, and let \(\bar{\delta}>0\) be given, then for every \(\delta\in[0,\bar\delta)\), define the family \(\kappa_\delta\in L^2(D\times D)\) of kernels as 
\begin{equation}\label{eq:pert_kernel}
    \kappa_\delta=\kappa_0+\delta\cdot \dot{\kappa}+ r_\delta
\end{equation}
where \(\dot{\kappa},\ r_{\delta }\in L^{2}(D\times D)\) with \(\norm{r_{\delta }}_{2}=o(\delta )\). 

 Finally, define the linear operators (Hilbert-Schmidt integral operators) \(\dot \cL,\ \cL_\delta: L^1\to L^1\) by 
\begin{equation}\label{eq:diff_to}
    \dot \cL f(y):= \int \dot \kappa(x,y)f(x)dx,
\end{equation}

\begin{equation}\label{pert_to}
     \cL_\delta f(y):= \int \kappa_\delta(x,y)f(x)dx.
\end{equation}
Let us assume that the operators \(\cL_\delta\) are integral preserving for every \(\delta\in [0,\bar\delta)\), that is, for each \(g\in L^1\)
\[
\int\cL_\delta g\ dm=\int g\ dm,
\]
where \(m\) is the Lebesgue measure.

\begin{lemma}\label{lem:int_pr}
    For \(\delta\in[0,\bar\delta)\), the operator \(\cL_\delta\) is integral preserving if and only if both \(\dot \kappa\) and \(r_\delta\) are zero average in in \(y\) direction, that is, for almost each $x$, \(\int \dot \kappa(x,y)dy= \int r_\delta(x,y)dy=0\).

\end{lemma}
\begin{proof} The converse part is straight forward, therefore we prove the forward implication.
    \(\cL_\delta\) is integral preserving if and only if, for every \(g\in L^1\)
    \begin{align*}
          \int g\ dm &= \int \cL_\delta g\ dm  \\
  & = \int \int \kappa_\delta(x,y)g(x)\ dx dy \\
  & = \int\int (\kappa_0+\delta\dot \kappa+r_\delta)(x,y) g(x)\ dxdy \\
  & = \int\int \kappa_0(x,y) g(x)\ dxdy + \int\int (\delta\dot \kappa+r_\delta)(x,y) g(x)\ dxdy\\
  & = \int\int g(x)\ dxdy + \int\int (\delta\dot \kappa+r_\delta)(x,y) g(x)\ dxdy
  \end{align*}
   thus, the above equality reduces to , for every \(g\in L^1\), 
   \[
   \int \int (\delta\dot \kappa+r_\delta)(x,y)g(x)\ dx dy = 0 = \int \int (\delta\dot \kappa+r_\delta)(x,y)g(x)\ dy dx.
   \]

Hence, for almost every \(x\in\bR^d\), 
\[0= \int (\delta\dot \kappa+r_\delta)(x,y)g(x)\ dy = \int \delta\dot \kappa(x,y)g(x)\ dy + \int r_\delta(x,y)g(x)\ dy.\]
   That is, 
   \[\int \dot \kappa(x,y)g(x)\ dy = -\int\frac{r_\delta}{\delta}(x,y)g(x)\ dy\]
   Taking limit \(\delta\to 0\) both sides, we get 
   \[
   \lim\limits_{\delta\to 0}\int\dot \kappa(x,y)g(x) \ dy=-\lim\limits_{\delta\to 0}\int\frac{r_\delta}{\delta}(x,y)g(x) \ dy =0
   \]
   Since so is true for all \(g\in L^1\), we get that for almost all \(x\in\bR^d\), \(\int\dot \kappa(x,y)dy=0\) which implies \(\int r_\delta(x,y)\ dy=0\).

\end{proof}

Now recall that by Proposition~\ref{prop:inv_den}, there exists an invariant density \(f_0\in\fB\) (the eigenvector corresponding to the eigenvalue 1) of the operator \(\cL_0\). Consequently, we have the following lemma.
\begin{lemma}\label{ldot2} 
The operators \(\dot \cL\) and \(\cL_\delta\) restricted to the strong space \(\fB\) satisfy the following limit
\begin{equation*}
\lim_{\delta \rightarrow 0}\norm{\frac{\cL_{\delta }-\cL_{0}}{\delta }f_{0}-%
\dot{\cL}f_{0}}_{s}=0
\end{equation*}
where \(f_0\) is the invariant density for the operator \(\cL_0\).
\end{lemma}

\begin{proof}
By definition 
\[
\begin{split}
\norm{\frac{\cL_{\delta }-\cL_{0}}{\delta }f_{0}-\dot{\cL}f_{0}}_{s} &=\norm{\int \frac{\kappa_{\delta }(x,y)-\kappa_{0}(x,y)}{\delta }f_{0}(x)\
dx-\int \dot{\kappa}(x,y)f_{0}(x)\ dx}_{s} \\
&=\norm{\int \frac{r_{\delta }(x,y)}{\delta }f_{0}(x)\ dx}_{s}\\
&=\norm{\int \frac{r_{\delta }(x,y)}{\delta }f_{0}(x)\ dx}_{L^1_2}
     + \norm{1_D\int \frac{r_{\delta }(x,y)}{\delta }f_{0}(x)\ dx}_{2}.
\end{split}
\]
Now using Holder's inequality and the fact that $r_{\delta }$ has compact support, there exists $C>0$ such that 
\[
\begin{split}
    \norm{\frac{\cL_{\delta }-\cL_{0}}{\delta }f_{0}-\dot{\cL}f_{0}}_{s}
    &\leq\norm{1_D\rho_2}_2\norm{\int \frac{r_{\delta }(x,y)}{\delta }f_{0}(x)\ dx}_{2}+     \norm{\int \frac{r_{\delta }(x,y)}{\delta }f_{0}(x)\ dx}_{2}\\
    &\leq(\norm{1_D\rho_2}_2+1)\norm{\int \frac{r_{\delta }(x,y)}{\delta }f_{0}(x)\ dx}_{2}\\
    &\leq (\norm{1_D\rho_2}_2+1)\norm{1_Df_{0}}_{2} \norm{\frac{r_{\delta }}{\delta }}_{2}\\
    &\leq  \frac{C}{\delta }\norm{1_Df_{0}}_{2}\rightarrow 0
\end{split}
\]
where \(\rho_2\) is as described in \eqref{WF}, \(C>0\) is some constant and the second last inequality comes from \eqref{KF}.
\end{proof}
The above lemma implies that the operator \(\dot\cL\) can be thought of as the derivative, with respect to \(\delta\), of the operator \(\cL_\delta\).
Now, we are interested to consider the linear response of our systems to such perturbations. But first, let's understand how close the perturbed operators \(\cL_\delta\) are to the initial operator \(\cL_0\).

\begin{lemma}\label{lem:close}
Let $\cL_{\delta }$ be a family of transfer operators associated with the kernels $\kappa_{\delta }$ as before. Then, there exist constants $\overline{\delta }, C\geq 0$ such that for each $\delta\in \lbrack 0,\overline{\delta })$ 
\begin{equation}
\norm{\cL_{0}-\cL_{\delta }}_{\fB\rightarrow L^{1}}\leq C\delta. 
\end{equation}
\end{lemma}

\begin{proof}
    Let \(f\in \fB\) be such that $ \norm{f}_s \leq 1$. Since \(\delta\dot \kappa+r_\delta\) is compactly supported on \(D\times D\), it implies that \((\delta\dot \kappa+r_\delta)f\) is also compactly supported in $D$. Using the fact that inside the set \(D\), by Cauchy Schwartz inequality, the norms satisfy \(\norm{\cdot}_{1}\leq C_1 \norm{\cdot}_{2}\), for some $C_1>0$, we have  
\[
\begin{split}
     \norm{(\cL_\delta-\cL_0)f}_{1}&= \norm{\int(\delta \dot{\kappa}+r_\delta)(x,y)f(x)dx}_{1}\\
    &\leq C_1 \norm{\int (\delta \dot{\kappa}+r_\delta)(x,y)f(x)dx}_{2}\\
    & =C_1\norm{\delta\dot \cL f + \int r_\delta(x,y)f(x)dx}_{2}\\
    & \leq \delta C_1\norm{\dot \kappa}_{2}\norm{f|_D}_{2}+C_1\norm{r_\delta}_{2}\norm{f|_D}_{2},
\end{split}
\]
where we have used that \(\dot{\kappa}, r_\delta\in L^2(D\times D)\) 
and in the last inequality we used \eqref{KF}. Since \(||r_\delta||_2=o(\delta)\), there exists \(C>0\) such that \(\norm{\delta C_1 \dot \kappa+ C_1 r_\delta }_{2}\leq C\delta\). Finally, using \(\norm{f|_D}_{2}\leq \norm{f}_s\leq 1\), we have the result.
\end{proof}

To obtain a result on linear response, we need the perturbed operators to possess certain properties, such as the Lasota-Yorke inequality, the existence of an invariant density, and a bounded resolvent. We prove these properties in the following lemmas.
 
\begin{lemma}\label{lem:LY_delta}
Let $\cL_{\delta }$ be the family of transfer operators associated with the kernels $\kappa_{\delta }$ as above. Then there exist constants $A,B,\overline{\delta }>0$, $\lambda<1$ such that for each $\delta \in \lbrack 0,\overline{\delta })$, $n\geq 0$ 
\begin{equation}
\norm{\cL_{\delta }^{n}f}_{s}\leq A\lambda ^{n}\norm{f}_{s}+B\norm{f}_{1} .
\end{equation}
\end{lemma}

\begin{proof}
We have that 
\begin{eqnarray*}
\norm{\cL_{\delta }f-\cL_{0}f}_{s} &= &\norm{\cL_{0}f+\delta
\int \dot{\kappa}(x,y)f(x)dx+\int {r_{\delta }(x,y)}f(x)\
dx-\cL_{0}f}_{s} \\
&= &\norm{\int (\delta \dot{\kappa}(x,y)+r_{\delta }(x,y))f(x)dx}_{s} \\
& \leq &\norm{\int \left( \delta \dot{\kappa}+ r_\delta \right) f(x) dx}_{L^1_2} + {\norm{1_D\int ( \delta \dot{\kappa}+ r_\delta ) f(x) dx}_{2}}\\
&\leq  &(\norm{1_D\rho_2}_2+1)\norm{\int(\delta \dot{\kappa}+ r_\delta)f(x)dx}_2.
\end{eqnarray*}%
 where in the last inequality we have used Holder's inequality and the fact that $\dot{\kappa}(x,y)$ and $ {r_{\delta }(x,y)}$ are supported only in $D$. Now, using \eqref{KF} and the fact that \(\norm{r_\delta}_2=o(\delta)\), there exists \( C_1>0\) such that
\begin{equation}\label{eq:diff_1}
    \norm{\cL_{\delta }f-\cL_{0}f}_{s} \leq \delta
C_{1}\norm{1_{D}f}_{{2}} \leq \delta C_{1}\norm{f}_{s}.
\end{equation}

Hence, there exists $C(n)\geq 0$ (depending on $n$) such that%
\begin{equation*}
\norm{\cL_{\delta }^{n}f-\cL_{0}^{n}f}_{s}\leq \delta
C(n)\norm{f}_{s}.
\end{equation*}

By Theorem \ref{lem:LY-ineq}
\begin{eqnarray*}
\norm{\cL_{\delta }^{n}f}_{s} &\leq&\norm{\cL_{\delta }^{n}f-\cL_0^n f}_{s}+\norm{\cL_{0 }^{n}f}_{s}\\
&\leq&\norm{\cL_{0}^{n}f}_{s}+\delta C(n)\norm{f}_{s} \\
&\leq &A\lambda ^{n}\norm{f}_{s}+B\norm{f}_{{1}}+\delta C(n)\norm{f}_{s}\\
&\leq &(A\lambda ^{n}+\delta C(n))\norm{f}_{s}+B\norm{f}_{{1}}.
\end{eqnarray*}%

Taking $n_{1}$ large enough, so that $A\lambda ^{n_{1}}\leq\frac{1}{3}$ and $\delta\in [0,\bar\delta)$ so small such that  $\delta C(n_1)\leq\frac{1}{3}$, we get 
\begin{equation*}
\norm{\cL_{\delta}^{n_{1}}f}_{s}\leq
\frac23\norm{f}_{s}+B\norm{f}_{{1}}.
\end{equation*}
Then, for every \(m\in \bN\), we get 
\[\norm{ \cL_{\delta }^{mn_{1}}f}_{s}\leq \left(\frac 23\right)^m\norm{
f}_{s}+(B+2)\norm{f}_{{1}}.\]
Similarly, for any \(n<n_1\), \(\norm{\cL^{n_1+n}_\delta f}_s\leq \frac{2}{3}(3\delta C_1+A\lambda)^n\norm{f}_s + C_2\norm{f}_{1}\), proving a uniform Lasota-Yorke inequality for $\cL_{\delta }.$
\end{proof}

\begin{proposition}\label{prop:inv_msr}
For the operators \(\cL_\delta\), and for $\delta $ small enough, there exists a unique invariant density $f_{\delta }\in \fB$ with $\int f_{\delta }=1$, that is, \(L_{\delta }f_{\delta }=f_{\delta }\).
  \end{proposition}
\begin{proof}
Firstly, recall that for all \(\delta\in[0,\delta)\), the operators \(\cL_\delta\) are integral preserving, which implies 1 lies in the spectrum of \(\cL_\delta\). Indeed, consider the dual operators \(\cL_\delta^*\); the fact that the operators are integral preserving implies that the Lebesgue measure is invariant for the dual operator \(\cL_\delta^*\). Accordingly, \(\cL_\delta^*1=1\), that is, 1 is in the spectrum of the dual operators and hence in the spectrum of the operators \(\cL_\delta\) for all \(\delta\geq 0\). \\
 Now, recall that by Lemma~\ref{lem:simple_ev}, 1 is the only eigenvalue of \(\cL_0\) on the unit circle.
Observe that for every \(\delta\in[0,\bar\delta)\), the operator \(\cL_\delta\) is a weak contraction on \(L_1\) and satisfies Lasota-Yorke inequality (by Lemma~\ref{lem:LY_delta}). Further, using Lemma~\ref{lem:close}, and the fact that \(\fB\) is compactly immersed inside \(L^1\) (Lemma~\ref{prop:cpt_incl}), we get, using \cite[Theorem 1]{KL}, that the isolated eigenvalues of \(\cL_\delta\) are arbitrarily close to that of \(\cL_0\), that is, the leading eigenvalue \(\lambda_\delta\) of \(\cL_\delta\) is also simple and satisfies that 
\[
\lim\limits_{\delta\to 0}\lambda_\delta = 1,
\] 
 which, in turn, implies that \(\lambda_\delta=1\) for small enough \(\delta\in[0,\bar\delta)\). Accordingly, there exists \(f_\delta\in \fB\) such that \(\cL_\delta f_\delta=f_\delta\) satisfying \(\int f_{\delta }=1\) being a probability density. Note that such an invariant density is unique. Indeed, if there exists another invariant density \(g_\delta\in\fB\), then note that \(\int(f_\delta-g_\delta)=\int f_\delta-\int g_\delta=1-1=0\), which implies 
 \[
 \int |f_\delta-g_\delta|=0=\norm{f_\delta-g_\delta}_1.
 \]
  Accordingly, using \eqref{KF2},
 \[
 \begin{split}
     \norm{f_\delta-g_\delta}_\infty
     &=\norm{\cL_\delta f_\delta-\cL_\delta g_\delta}_\infty\\
     &= \norm{\cL_\delta(f_\delta-g_\delta)}_\infty\\
     &\leq \norm{\kappa_\delta}_\infty\norm{f_\delta-g_\delta}_1\\
     &\leq \norm{\kappa_\delta}_\infty\norm{f_\delta-g_\delta}_1\\
     &= 0.
 \end{split}
 \]
 Hence the uniqueness.
\end{proof}

To be able to state the linear response result, the last ingredient we need is the bounds on the resolvent of perturbed operators, and hence the following lemma. Recall the definition of \(V_\fB\) and \(V_{L^1}\) given by \eqref{V_B} and \eqref{V_s} respectively.

\begin{lemma}\label{lem:pert_res}
For \(\delta>0\) small enough, the resolvent operators of the perturbed operators \(\cL_\delta\), given by \((Id-\cL_\delta)^{-1}\), are well-defined and satisfy
\[\norm{(Id-\cL_\delta)^{-1})}_{V_\fB\to V_{L^1}}<+\infty\]
and 
\[
\lim\limits_{\delta \to 0}\norm{(Id-\cL_\delta)^{-1}-(Id-\cL_\delta)^{-1}}_{V_\fB\to V_{L^1}}=0.
\]
\end{lemma}
\begin{proof}
    Note that the operators \(\cL_\delta\) are weak contractions on \(L_1\), satisfy Lasota-Yorke inequality (by Lemma~\ref{lem:LY_delta}), and, by the fact that \(\norm{\cL_0-\cL_\delta}_{\fB\to L^1}\leq C\delta\) (by Lemma\ref{lem:close}), we can use \cite[Theorem 1]{KL} to get the result.
\end{proof}
 Finally, we can state a response result adapted to our kind of systems and perturbations.

\begin{theorem}\label{thm:res_diff}
     For the operators \(\cL_\delta\), for $\delta >0$ small enough, we have linear response, that is
\begin{equation*}
\lim_{\delta \rightarrow 0}\norm{\frac{f_{\delta }-f_{0}}{\delta }
-(Id-\cL_{0})^{-1}\dot{\cL}f_{0}}_{1}=0.
\end{equation*}
\end{theorem}

Though in this work we consider operators which are not necessarily positive, the proof of the above statement is similar to many other linear response results (see e.g. \cite{FG23} Theorem 12, Appendix A).
We include it for completeness.

{
 \begin{proof}
 By Proposition~\ref{prop:inv_msr}, we know that for each \(\delta\in[0,\bar\delta)\), the operator \(\cL_\delta\) has a fixed point \(f_\delta\). Accordingly, we get
 \[
 \begin{split}
     (Id-\cL_\delta)\frac{f_\delta-f_0}{\delta}&=\frac{f_\delta-f_0}{\delta}-\frac{\cL_\delta f_\delta-\cL_\delta f_0}{\delta}\\
     &= \frac{-f_0+\cL_\delta f_0}{\delta}\\
     &= \frac{1}{\delta}(\cL_\delta-\cL_0)f_0.
 \end{split}
 \]    
 By the preservation of integral, for each $\delta ,$ $L_{\delta }$
 preserves $V_{\fB}$. Since for all $\delta >0$, $\frac{L_{\delta }-L_{0}}{\delta }f_{0}\in V_{\fB}$ and by Lemma~\ref{lem:pert_res}, for $\delta $ small enough, $(Id-L_{\delta})^{-1}:V_{\fB}\rightarrow V_{L^1}$ is a uniformly bounded operator. Accordingly, we can apply the resolvent to both sides of the expression above to get
 \begin{align}\label{xx1}
 (Id-\cL_\delta )^{-1}(Id-\cL_\delta )\frac{f_\delta -f_0}{\delta }
 &=(Id-\cL_\delta )^{-1}\frac{\cL_\delta -\cL_0}{\delta }f_{0}  
\\
 &=[(Id-\cL_{\delta })^{-1}-(Id-\cL_{0})^{-1} + (Id-\cL_{0})^{-1}]\frac{\cL_{\delta }-\cL_{0}}{\delta }f_{0}.
 \end{align}
 Since, by Lemma~\ref{lem:pert_res}, $\norm{(Id-\cL_{\delta })^{-1}-(Id-\cL_{0})^{-1}}_{V_{\fB}\rightarrow V_{L^1}}\rightarrow 0$, we have 
 \begin{align*}
 \norm{[(Id-\cL_\delta)^{-1}-(Id-\cL_0)^{-1}]\frac{\cL_\delta-\cL_0}{\delta}f_0}_{L^1} &\leq \norm{(Id-\cL_\delta)^{-1}-(Id-\cL_0)^{-1}}_{V_\fB\to V_{L^1}}\norm{\frac{\cL_\delta-\cL_0}{\delta}f_0}_s\\
 &\rightarrow 0.
 \end{align*}
 Further, Lemma~\ref{ldot2} implies $\lim_{\delta \rightarrow 0}\frac{\cL_\delta-\cL_0}{\delta }f_0$
 converges in $V_\fB,$ then \eqref{xx1} implies that
 \begin{align*}
 0&=\lim_{\delta \rightarrow 0}\norm{\frac{f_\delta-f_0}{\delta } -(Id-\cL_0)^{-1}\frac{\cL_\delta-\cL_0}{\delta }f_0}_{L^1} \\
 &=\lim_{\delta \rightarrow 0}\norm{\frac{f_\delta-f_0}{\delta }-(Id-\cL_0)^{-1}(\dot{\cL}f_0)}_{L^1}.
 \end{align*}
 \end{proof}
}

\section{Optimisation of response}\label{sec:optimal}

 \subsection{Optimization of the expectation of an observable} \label{sec:opt_exp} 
Consider the set \(P\subset L^2(D\times D)\) of allowed infinitesimal perturbations $\dot{\kappa}$ (as described in Section~\ref{sec:perturb}) to the kernel $\kappa_0$. Fix an observable $\phi\in L^\infty$; we are interested in finding, from the set of allowed perturbations, an \emph{optimal} perturbation $\dot{\kappa}_{opt}$ which maximises the rate of change of the expectation of \(\phi\). To perform the optimisation, we assume that \(P\) is a bounded, closed and convex subset of the Hilbert space \( L^2(D\times D)\). \\
We believe that the above hypotheses on \(P\) are natural; convexity is so because if two different perturbations of a system are possible, then their convex combination -- applying the two perturbations with different intensities -- should also be possible.\\

Let $\phi\in L^{\infty}(\bR^d,\mathbb{R})$ be an observable. If $\phi$ were the indicator function of  a certain set, for example, one could control the invariant density towards the support of $\phi$. \\
Given a family of kernels $\kappa_{\delta }=\kappa_0+\delta\dot{\kappa}+r_\delta$ (as in Lemma~\ref{ldot2}) with associated transfer operators \(\cL_\delta\) and invariant densities $f_{\dot{\kappa},\delta }$, we denote the response of the system to $\dot{\kappa}$ by 
\begin{equation}\label{R}
R(\dot{\kappa})=\lim\limits_{\delta \rightarrow 0}\frac{f _{\dot{\kappa},\delta}-f _{0}}{\delta }.
\end{equation}%
This limit converges in $L^{1}$ as proved in Theorem \ref{thm:res_diff}.
\begin{lemma}\label{lem:R_cty}
    The operator \(R:L^2(D\times D)\to L^1\) defined in \eqref{R} is continuous.
\end{lemma}
\begin{proof}

        Using Theorem~\ref{thm:res_diff}, we have 
    \[
    R(\dot{\kappa})=\lim\limits_{\delta \rightarrow 0}\frac{f _{\dot{\kappa},\delta}-f _{0}}{\delta }=(Id-\cL_0)^{-1}\dot{\cL}f_0,
    \]
   which implies 
    \[
    \norm{R(\dot{\kappa})}_{1}\leq \norm{R(\dot{\kappa})}_{s}\leq \norm{(Id-\cL_0)^{-1}}_{s}\norm{\dot\cL f_0}_{s}.
    \]
     Thanks to Proposition \ref{prop:bdd_res}, the operator $(Id- \mathcal{L}_0)^{-1}$ is bounded, and thus $\norm{(Id- \mathcal{L}_0)^{-1}}_s$ is finite. Consider the following 
    \[
    \begin{split}
        \norm{\dot\cL f_0}_{s}
        &= \norm{\int\dot \kappa(x,y)f_0(x)dx}_{s}\\
        &= \norm{\int\dot \kappa(x,y)f_0(x)dx}_{L^1_2}+\norm{1_D\int\dot \kappa(x,y)f_0(x)dx}_{2}.
    \end{split}
    \] 
    Using Holder's inequality and the fact that \(\dot \kappa\) has compact support, we get
    \[
    \norm{\int\dot \kappa(x,y)f_0(x)dx}_{L^1_2}\leq \norm{1_D\rho_2}_2\norm{\int\dot \kappa(x,y)f_0(x)dx}_{2},
    \]
    where \(\rho_2\) is the same as given in \eqref{WF}.
    Finally, using \eqref{KF2} and again the fact that \(\dot \kappa\) has compact support, there exists \(C\geq 0\) such that 
    \[
    \begin{split}
        \norm{\dot\cL f_0}_{s}
        &\leq (\norm{1_D\rho_2}_2+1)\norm{\int\dot \kappa(x,y)f_0(x)dx}_{2}\\
        &\leq (\norm{1_D\rho_2}_2+1) \norm{1_Df_0}_{2}\norm{\dot \kappa}_{2}\\
        &\leq C \norm{1_Df_0}_{2}\\
        &\leq C \norm{f_0}_s.
    \end{split}
    \]
     The above calculation implies the operator \(R\) is bounded and hence continuous.
\end{proof}

 Under our assumptions, since $\phi\in L^\infty$, we easily get 
\begin{equation}\label{LRidea2}
\lim_{\delta \rightarrow 0}\frac{\int \ \phi(x)f_\delta (x) \ dx-\int \
\phi(x)f_{0}(x) \ dx}{ \delta }=\int \ \phi(x)R(\dot{\kappa})(x)\ dx.  
\end{equation}

Hence the rate of change of the expectation of $\phi$ with respect to $\delta$ is given by the linear response of the system under the given perturbation. To take advantage of the general results of Appendix~\ref{appendix}, we perform the optimisation of $\dot{\kappa}$ over the closed, bounded and convex subset \(P\) of the Hilbert space $\fH=L^2(D\times D)$ containing the zero perturbation. To maximise the RHS of \eqref{LRidea2} we set 
\[
\mathcal{J}(\dot \kappa):=-\int \phi(x)\cdot R(\dot{\kappa})(x)\ dx
\]
and consider the following problem:
\begin{problem}\label{prob:max}
Find the  solution $\dot{\kappa}_{opt}$ to 
   \begin{eqnarray}  
   \max \left \lbrace  \mathcal{J}(\dot{\kappa}) \ : \  \dot{\kappa} \in L^2 (D\times D), \ \norm{\dot{\kappa}}_2 \leq 1\right \rbrace .
\end{eqnarray} 
\end{problem}

To answer the above problem, we state a rather general following result:

\begin{proposition}\label{prop:gen_cnvx}
Let \(P\subset L^2\) be a closed, bounded and strictly convex set such that its relative interior contains the zero vector. If $\mathcal{J}$ is not uniformly vanishing on \(P\), then the problem 
$$
     \max \left \lbrace \mathcal{J}(\dot{\kappa})\ : \   \dot{\kappa} \in P\right \rbrace 
$$
has a unique solution.
\end{proposition}
\begin{proof}
     The linearity of \(R\) from Lemma~\ref{lem:R_cty} implies that the function \(\dot \kappa\mapsto \int \phi(x)R(\dot \kappa)(x)dx\) is continuous. Since \(P\) satisfies the hypothesis, we can apply Proposition~\ref{prop:exist} and Proposition~\ref{prop:uniqe} to get the result.
\end{proof}

Finally, to address Problem~\ref{prob:max}, note that the perturbations \(\dot \kappa\) lie in the closed unit ball of \(L^2(D\times D)\), which is a closed, bounded and a strictly convex set having the zero vector in its interior. Thus, we get the unique kernel perturbation \(\dot{\kappa}_{opt}\) using Proposition~\ref{prop:gen_cnvx}. 

\subsection{Numerical scheme for the approximation of the optimal perturbation}
\label{sec: numerical scheme for the approximation of the opt. pertb.}
In this section, we present a simple numerical recipe to obtain a Fourier approximation of the optimal perturbation $\dot{\kappa}_{opt}$ in Problem \ref{prob:max}. Similar methods, based on  Fourier approximation on suitable Hilbert spaces  were used in \cite{FG23}, and \cite{GN25}.
 
 By Lemma \ref{lem:R_cty}, and the fact that $\phi\in L^\infty$, we obtained in the proof of Proposition \ref{prop:gen_cnvx} that $\mathcal J:L^2(D\times D)\to \mathbb{R}$ is continuous. 
 Thus, $L^2(D\times D)$ being a Hilbert space, by  using Riesz's representation theorem, there exists \(g\in L^2\) such that 
\[
\mathcal{J}(\dot{\kappa})= \langle g, \dot{\kappa} \rangle,
\]
which implies that 
\[
\dot{\kappa}_{opt}=\frac{g}{\norm{g}}_{2}.
\]

Now we state how to approximate \(g\) by computing its Fourier coefficients. Consider an orthonormal basis
\(\{h_i\}_{i\in \bN}\) of\footnote{One can consider the basis consisting of indicator functions, or that of Hermite polynomials etc.} \(L^2(D\times D)\) and define
\[
G_i = \langle g, h_i \rangle = \mathcal{J}(h_i) = \int \phi~dR(h_i).
\]
 We have then that the sequence \(\Sigma_{i=0}^n G_ih_i\) converges to \(g\) as \(n\to\infty\).\\

\section{Numerical experiments}\label{sec:numerics}
In this section, we present numerical experiments related to Problem~\ref{prob:max}.

We find the optimal infinitesimal perturbation that maximises the expected value of an observable, as discussed in Section \ref{sec:opt_exp}. 
We will consider a suitable SDE on the real line and on this simple example we  search for the optimal perturbation in order to maximize the increase of a certain observable.

Specifically, in Section \ref{sec: sde and fpe} we present and describe the details of the gradient SDE that will express the dynamics of each of our experiments. Section \ref{sec: fd for fpe} is aimed at describing a classical finite difference method used to solve the partial derivative equation of the parabolic type that identifies the solution law of the SDE. Section \ref{sec: numerical scheme optimal problem} presents how to numerically find the solution to the problem presented in Section \ref{sec: numerical scheme for the approximation of the opt. pertb.}. Finally, Section \ref{sec: results} presents and explains the results of our experiments. They consist of taking observables as probability density functions of Gaussian variables: symmetric with respect to the domain in the first case, and asymmetric in the second. 

\subsection{Gradient system and Fokker-Planck equation}
\label{sec: sde and fpe}
We consider, for a noise intensity $\varepsilon>0$ and a final time $T>0$, the SDE given by
\begin{equation}
\begin{split}
    dY_t^x &= - V'(Y_t^x)dt + \varepsilon dW_t, \quad t \in (0,T)\\
    Y_0^x &= x,
\end{split}
\label{eq: SDE gradient type}
\end{equation}
where $V(y) = \frac{y^4}{4}- \frac{y^2}{2}$ is a symmetric double well potential. $(W_t)_t$ denotes a Brownian Motion (BM) and $x \in \mathbb{R}$ is a deterministic initial condition. The final time $T>0$ will be the time in which we investigate the optimal response in the following sections. {In other words, we will numerically investigate the properties of the the transfer operator associated with the evolution of the SDE from time $t=0$ to time $t = T$.}

This SDE, which belongs to the class of gradient systems, is often used to describe bistability in many applications, such as phase separation in physics or abrupt climate shifts in paleoclimate studies. If $\varepsilon = 0$, the dynamical system given by \eqref{eq: SDE gradient type} presents three fixed points: $\pm 1$ and $0$. The former are asymptotically stable and correspond to the minimum points of $V$; the latter is unstable, since it corresponds to a local maximum point of $V$.

Considering the case $\varepsilon>0$, the existence and uniqueness of a solution is classical. Indeed, since the drift term $y \mapsto y - y^3$ is locally Lipschitz, the SDE \eqref{eq: SDE gradient type} admits, locally in time, an almost everywhere (a.e.) continuous strong solution, see \cite{Baldi2017} for more details. Further, since the potential $V$ is coercive, finite time explosion of the solution can be ruled out (\cite[Theorem 3.5]{Mao1997}), resulting in the existence and uniqueness of a strong solution $(Y_t^x)_{0 \leq t \leq T}$, for any $T>0$. Furthermore, since the potential $V$ is regular and coercive, Hormander's Theorem assures that the law $\mathcal{L}_{Y_t^x}$ of $Y_t^x$ has a $C^\infty$ density $y \mapsto p^x(y,t)$ with respect to the Lebesgue measure on $\mathbb{R}$, see \cite{Hormander1967} and \cite[Section 7]{Hairer2010}. It satisfies, in the weak sense, the Fokker-Planck equation (FPE) (\cite{Risken1984,Otto1998, Gardiner2009})
\begin{equation}
\left \lbrace
\begin{aligned}
     \partial_t p^x  -\frac{\varepsilon^2}{2} \Delta_y p^x + \partial_y \left( pb\right) &= 0, &x\quad &(y,t) \in  \mathbb{R} \times (0,T), \\
    p(0,y) & = \delta_x (y), &\quad & y \in \mathbb{R}, 
\end{aligned}
\right. 
\label{eq: fpe}
\end{equation}
where $b(y) = - V'(y) = y-y^3$ is the drift term  of the SDE \eqref{eq: SDE gradient type} and $\delta_x$ denotes the Dirac delta concentrated in $x$. Note that, given the final time $T>0$, which is the same time at which we investigate the linear response, the kernel $\kappa(x,y)$ defining the transfer operator $\mathcal{L}_0$ is given by
$$
\kappa(x,y) := p^x(y,T).
$$
A numerical approximation of the kernel is represented in Figure \ref{subfig: kernel}, obtained for $T = 1$ and $ \varepsilon = 0.25$. Given an initial condition $x$, the probability density function (PDF) $p$ is concentrated around either $1$ or $-1$, depending on the sign of the initial condition. This bistability is clear also in Figure \ref{subfig: invarian density f_0}, which depicts the invariant density $f_0$ for $\mathcal{L}_0$. In fact, it is symmetric, with two maximum points in $\pm 1$. In the next section, we are going to describe our numerical approximation for the solution of the FPE and how to use it to study the linear response problem.

\subsection{Finite difference method for the Fokker-Planck equation}
\label{sec: fd for fpe}
We apply an implicit finite-difference (FD) method (\cite{Quarteroni, Thomas1995}) to approximate the solution of the parabolic problem \eqref{eq: fpe} in the finite domain $\Omega \times (0,T) = (-a,a) \times (0,T)$, with the constraint to choose a sufficiently large $a>0$ such that $p$, which describes a PDF, is negligible on $\Omega ^c$. The FPE \eqref{eq: fpe} is a second-order parabolic PDE, and to consider a solution, we need to complete it with boundary conditions at the end of the domain $\Omega$. We choose reflecting boundary conditions, which means that no probability mass can escape the domain $\Omega = (-a,a)$ (\cite{Gardiner2009}). This leads to 
\begin{equation}
\left \lbrace
\begin{aligned}
     \partial_t p^x - \frac{\varepsilon^2}{2} \Delta_y p^x +\partial_y \left( pb\right)&= 0, &\quad &(y,t) \in  \Omega \times (0,T), \\
 \frac{\varepsilon^2}{2} \partial_y p(-a,t) -p(-a,t) b(-y)&=0, &\quad &t>0, \\
\frac{\varepsilon^2}{2} \partial_y p(a,t) - p(a,t) b(y)  & =0,  &\quad & t>0,\\
    p(0,y) & = \delta_x (y), & \quad & y \in \Omega .
\end{aligned}
\right.
\label{eq: FPE with reflecting boundary conditions}
\end{equation}

\begin{figure}
\begin{subfigure}[h]{0.49\linewidth}
\includegraphics[width=\linewidth]{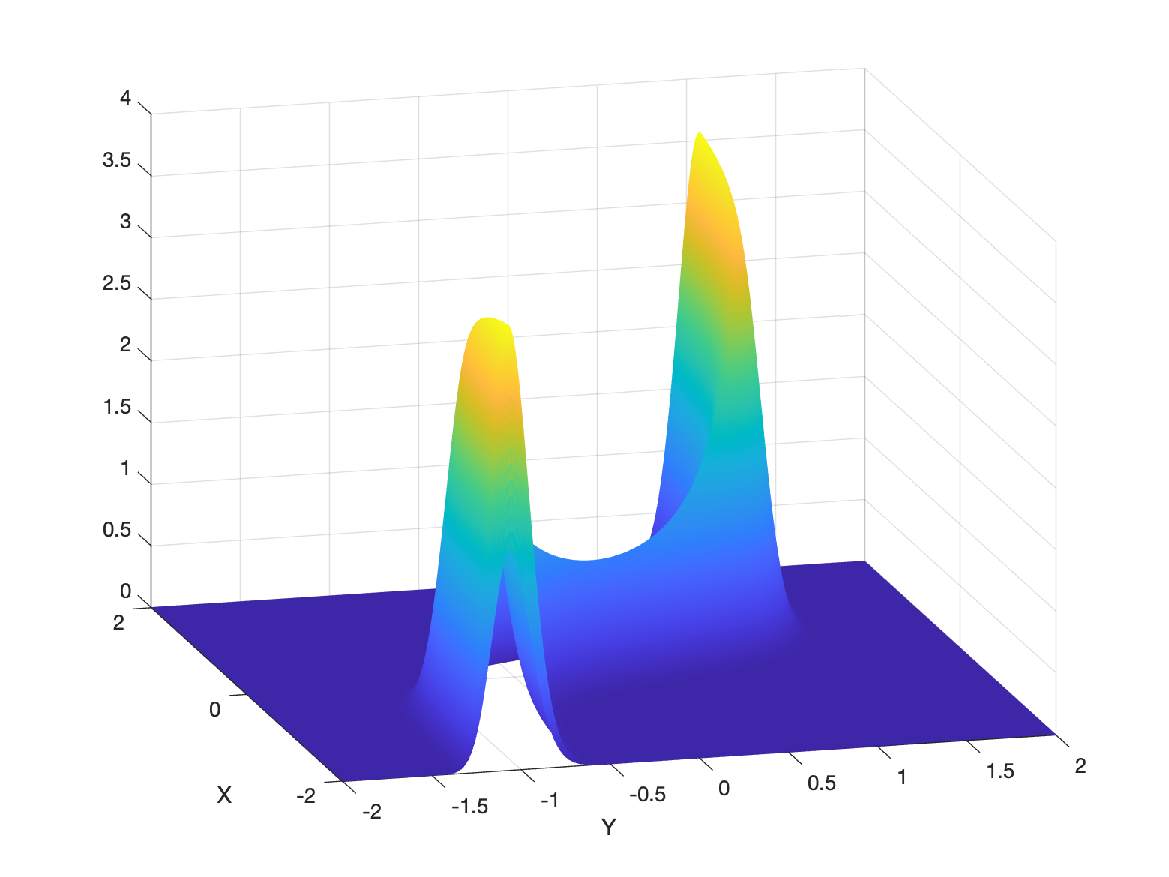}
\caption{Kernel $(x,y) \mapsto \kappa(x,y)$}
\label{subfig: kernel}
\end{subfigure}
\hfill
\begin{subfigure}[h]{0.49\linewidth}
\includegraphics[width=\linewidth]{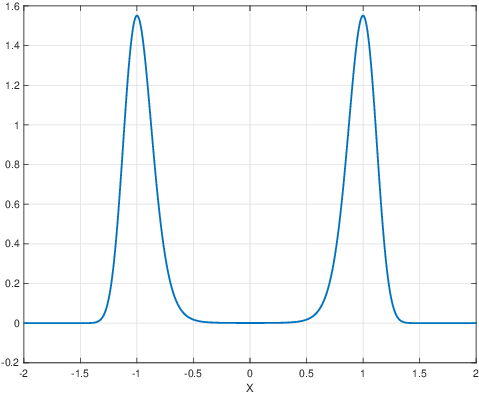}
\caption{Invariant density for $\mathcal{L}_0$}
\label{subfig: invarian density f_0}
\end{subfigure}%
\caption{Numerical approximation for (a) kernel $\kappa(x,y) = p^x(y,T)$, where $p$ is a solution of the FPE \eqref{eq: FPE with reflecting boundary conditions}, and (b) invariant density for the transfer operator $\mathcal{L}_0$. The final time is $T =1,$ the noise intensity is $\varepsilon = 0.25$, the domain is $\Omega = (-2,2)$, while the mesh sizes are $\Delta x = \Delta x = 2\cdot 10^{-3}$.  }
\end{figure}

First, we consider a uniform mesh of the space and time domain, i.e., we fix
$$
y_i := -a + i \Delta y, \quad i =0,...,m,\quad  \Delta y = \frac{2a}{n},
$$
and
$$
t_j := j \Delta t, \quad j = 0,...,m, \quad  \Delta t = \frac{T}{m},
$$
where $n,m >0$ determine the number of points in the meshes. We denote by
$$
p_{ij} := u(y_i, t_j), \quad i = 0,...,n, \; j = 0,...,m,
$$
the numerical approximation for the solution of \eqref{eq: FPE with reflecting boundary conditions}. Similarly, we set
$$
b_i = b(x_i), \quad i = 0,...,n.
$$

Second, we use the centered formula to approximate the second derivative, i.e.
$$
\Delta_y p_{\mid (y,t) = (y_i, t_j)} \approx \frac{p_{i-1,j+1}- 2p_{i,j+1}+p_{i+1,j+1}}{\Delta x^2},
$$
and a centered difference for the first derivative
$$
\partial_y(pb)_{\mid (y,t) = (y_i, t_j)} \approx \frac{p_{i+1,j+1}b_{i+1}-p_{i-1,j+1}b_{i-1}}{2\Delta x}.
$$
Note that in the previous equations, when we are at the boundary of $\Omega$, there appear the terms $p_{-1,j}, p_{n+1,j}, b_{-1},$ and $ b_{n+1}  $. The former are just auxiliary unknowns, since the solution $p$ of the FPE is not defined outside $\Omega \times (0,T)$; the latter are defined as
$$
b_{-1} = b(y_{-1}), \quad b_{n+1} = b(y_{n+1}).
$$
In addition to this, to approximate the Dirac Delta we consider the same mesh for the initial condition domain, i.e., we set
$$
x_l = y_l, \quad l = 0,...,n.
$$
Furthermore, for each $l =0,...,n$ we approximate the Dirac Delta $\delta_{x_i}$ with the PDF $f_{\mu_l, \sigma^2}$ of a Gaussian random variable with mean $\mu_l = x_l$ and standard deviation $\sigma = 10^{2} \cdot \Delta x$.

In conclusion, for each $l >0$, the implicit FD that we apply can be written as
\begin{align}
\begin{cases}
        \frac{p_{i,j+1}-p_{i,j}}{\Delta t} &= \frac{\varepsilon}{2\Delta x^2} \left(p_{i-1,j+1}- 2p_{i,j+1}+ p_{i+1,j+1} \right)\\
        &-\frac{p_{i+1,j+1}b_{i+1} - p_{i-1,j+1}b_{i-1}}{2\Delta x}, \;i =  0,...,n,\; j = 1,..., m, \\
         0 & =  
        \frac{\varepsilon^2}{2} \left(\frac{p_{1,j}- p_{-1,j}}{2\Delta x} \right) - p_0b_0, \quad j =1,...,m  \\
        0&= 
        \frac{\varepsilon^2}{2} \left(\frac{p_{n+1,j}- p_{n-1,j}}{2\Delta x} \right) - p_nb_n , \quad j =1,...,m\\
        p_{i,0} &= f_{\mu_l, \sigma }(y_i), \quad i =0,...,n.
            \label{eq: implicit FD FPE}
\end{cases}
    \end{align}

\subsection{Numerical scheme for the optimal response problem}
\label{sec: numerical scheme optimal problem}
In this section, we describe how to numerically approximate the solution of
\begin{equation}
\min \left \lbrace \mathcal{J}(\Dot{\kappa}) \; \mid \; \dot \kappa \in L^2 (D),\;  \norm{{\dot \kappa}}_{L^2(D)}  \leq 1\right \rbrace,
\label{eq: optimal problem}
\end{equation}
where $D = [-d,d] \subsetneq \Omega$. We know that, given an observable $\phi \in L^\infty$, the optimal perturbation is given by
$$
g = \frac{\sum_{r=0}^{+\infty} G_r h_r}{\left(\sum_{r = 0}^{+\infty} G_r^2 \right)^{1/2}},
$$

where $(h_r)_r$ denotes an orthonormal basis of $L^2(D\times D)$ and
$$
G_r = \int_D \phi(y) (Id- \mathcal{L}_0)^{-1} \int_D h_r(x,y) f_0(x) dx dy.
$$
Further, $f_0$ denotes the invariant density for $\mathcal{L}_0$. In particular, we consider as basis of $L^2(D \times D)$ the sine-cosine wavelets $\mathcal{B}$ defined as
\begin{equation*}
    \begin{split}
        \mathcal{B} = \bigcup_{i\geq 1, j \geq 0} & \left \lbrace 
\frac{1}{\sqrt{2d}} \cos( i \frac{\pi }{d}x) \cdot \cos (j \frac{\pi}{d}y), \frac{1}{\sqrt{2d}} \cos( i \frac{\pi }{d}x) \cdot \sin (j \frac{\pi}{d}y), \right. \\
& \left. \frac{1}{\sqrt{2d}} \sin( i \frac{\pi }{d}x) \cdot \cos (j \frac{\pi}{d}y), \frac{1}{\sqrt{2d}} \sin( i \frac{\pi }{d}x) \cdot \sin (j \frac{\pi}{d}y)
\right \rbrace .
\end{split}
\end{equation*}
Note that we avoid inserting the wavelets with index $i=0$ in the basis $\mathcal{B}$ defining the vector space of the allowed perturbation, since we impose that any perturbation $\dot \kappa$ needs to satisfy $\int_D \dot \kappa(x,y) dx = 0$. Then, for each $r$, we approximate the coefficient $G_r$ as follows.

Let $I_1 = \left \lbrace x_i \mid x_i \in D \right \rbrace = \left \lbrace x_i'
 \right \rbrace_{i=0,...,n_1}$ and $ \vert I_1 \vert = n_1 +1$. We adopt the same notation for $I_2 = \left \lbrace y_j \mid y_j \in D \right \rbrace = \left \lbrace y_j'
 \right \rbrace_{j=0,...,n_1}$. First, the definite integral defining $G_r$ is approximated by using the composite Simpson's $1/3$ rule (\cite{Quarteroni2007}), which leads to
\begin{equation}
    \begin{split}
        \int_{-d}^d \eta_1(x) dx &\approx  S_{1/3}(\Vec{\eta}_1) =: \frac{\Delta x}{3} \left( \eta_1(y_0') + 4 \eta_1(y_1') + 2\eta_1(y_2') + 4 \eta_1(y_3') +2\eta_1(y_4')+\cdots \right. \\
        &\left.
        +2 \eta_1(y_{n-2}') + 4 \eta_1(y_{n-1}') + \eta_1(y_{n_{1}}') \right),
    \end{split}
\end{equation}
with 
$$
\eta_1(y_j ') =\phi(y_j') \cdot (Id- \mathcal{L}_0)^{-1} \int_D h_r(x,y_j')f_0(x) dx :=\phi(y_j') \cdot \eta_2(y_j').
$$

Second, the vector $\Vec{\eta}_2 = (\eta_2(y_j'))_j \in \mathbb{R}^{n_1 +1}$ is the solution of the linear system
$$
\left(Id_{n_1 +1} - L_0^T \right) \Vec{\eta}_2 = \Vec{d}.
$$
The matrix $L_0 = (l_{ij})_{ij} \in \mathbb{R}^{(n_1+1)\times (n_1+1)}$ represents the numerical approximation, via the same Simpson rule recalled before, of the operator $\mathcal{L}_0$. Its elements are given by
$$
l_{ij} = \begin{cases}
    \frac{\Delta x}{3} \kappa (x_i', y_j') \quad &\text{ if }i=0, j \geq 0, \\
    4\frac{\Delta x}{3} \kappa (x_i', y_j') \quad &\text{ if }i \text{ is odd}, j \geq 0, \\
    2\frac{\Delta x}{3} \kappa (x_i', y_j') \quad &\text{ if }i \text{ is even}, j \geq 0.
\end{cases}
$$
Further, the constant term vector $\Vec{d} =(d_j)_j \in \mathbb{R}^{n_1 +1}$ is obtained by approximating the innermost integral of $G_r$ as follows
$$
d_j = \int_D h_r(x, y_j) f_0(x) dx \approx  S_{1/3}(\Vec{\eta}_3),
$$
with $\Vec{\eta}_3^{\,(j)} = (\eta_3^{\,(j)}(x_i))_i \in \mathbb{R}^{n_1 +1}$ and $\eta_3^{\,(j)}(x_i) = h_r(x_i', y_j') \cdot f_0(x_i').$

Third, the invariant density $f_0$ is obtained by computing the left eigenvector $\Vec{f}_0 = (f_{0,i})_{i = 0,...,n}$ corresponding to the eigenvalue $1$ of the matrix
$$
K \cdot dx,
$$
where the entries of the matrix $K = (\kappa_{ij})_{ij} \in \mathbb{R}^{(n+1) \times (n+1)}$ are given by
$$
\kappa_{ij} = \kappa(x_i, y_j).
$$
Indeed, this is equivalent to find a vector $\Vec{f}_0 \in \mathbb{R}^{n+1}$ such that
$$
\Vec{f}_0 \cdot K \cdot dx = \Vec{f}.
$$
Note that the left-hand side is an approximation of the transfer operator $\mathcal{L}_0$ evaluated on $f_0$. In fact, the previous equation, component-wise, reads as
$$
dx \cdot \sum_{i} \kappa(x_i', y_j') f_{0,i} = f_j, \quad \forall j = 0,..., n, $$ with the left side being an approximation of $\mathcal{L}_0f$. The existence of a such eigenvector for the eigenvalue $1$ is a consequence of the Perron-Frobenius Theorem for row stochastic matrices, see \cite{Bini2005}.

Lastly, the sum defining $g$ is truncated, and we consider the approximation of $g$ given by the elements in the truncated basis $B_{I,J}$ defined as 
\begin{equation*}
    \begin{split}
        \mathcal{B}_{I,J} = \bigcup_{i= 1,...,I} \bigcup_{j=0,...,J} & \left \lbrace 
\frac{1}{\sqrt{2d}} \cos( i \frac{\pi }{d}x) \cdot \cos (j \frac{\pi}{d}y), \frac{1}{\sqrt{2d}} \cos( i \frac{\pi }{d}x) \cdot \sin (j \frac{\pi}{d}y), \right. \\
& \left. \frac{1}{\sqrt{2d}} \sin( i \frac{\pi }{d}x) \cdot \cos (j \frac{\pi}{d}y), \frac{1}{\sqrt{2d}} \sin( i \frac{\pi }{d}x) \cdot \sin (j \frac{\pi}{d}y)
\right \rbrace .
\end{split}
\end{equation*}
Thus, our approximation of the solution $g$ is given by
$$
g_{I,J} = \frac{\sum_{r \, : h_r \in \mathcal{B}_{I,J}} G_r h_r}{\left(\sum_{r \, : h_r \in \mathcal{B}_{I,J}} G_r^2 \right)^{1/2}}   .
$$

\subsection{Results}
\label{sec: results}
In this section, we present our numerical experiments to approximate the solution of the optimal control problem \eqref{eq: optimal problem} using two different observables, \( \phi \). Both observables are selected from the class of PDFs \( f_{\mu, \sigma^2} \) of Gaussian random variables with mean \( \mu \) and variance \( \sigma^2 \).

In the first experiment, we use a symmetric observable with \( \mu = 0 \) and \( \sigma = 0.1 \). In the second experiment, we use an asymmetric observable with \( \mu = -0.5 \) and \( \sigma = 0.1 \). Similar results can be achieved by varying \( \mu \), \( \sigma \), and the class of the observable (such as continuous bump functions or polynomials). However, using non-continuous observables (like indicator functions) may introduce minor numerical errors due to the finite selection of basis elements. We will discuss the results of the symmetric experiment, shown in Figure \ref{fig: symmetric experiment}, in detail. The results for the asymmetric case, shown in Figure \ref{fig: asymmetric experiment}, follow the same explanation and are not repeated here.

First, Figure \ref{fig: observable symmetric} depicts the symmetric observable used in the first experiment. Figure \ref{fig: optimal perturbation symmetric} represents the approximation $g_{I,J}$ of the optimal perturbation $g$ that solves the problem \eqref{eq: optimal problem}. This optimal perturbation, \( g_{I,J}\), provides the infinitesimal adjustment to the kernel needed to maximize \( \int_D \phi(y) f_{\delta}(y) \, dy \). It is worth pointing out that the perturbation \( g_{I,J} \) is significant (i.e., far from zero) only where the observable is significant. Furthermore, for any given value \( y \) where \( g_{I,J} \) is significant, for example, \( y = 0 \), the restriction of the perturbation \( x \mapsto g_{I,J}(x,y) \) redistributes mass from areas where the invariant density \( f_0 \) is small (away from \( x = \pm 1 \)) to areas where the invariant density is large (around \( x = \pm 1 \)). Additionally, due to the choice of basis $\mathcal{B}_{I,J}$, it can be numerically verified that \( \int_D g_{I,J}(x,y) \, dx = 0 \) for any \( y > 0 \).

Figure \ref{fig: perturbed kernel symmetric} visually represents the perturbed kernel \( \kappa_\delta = \kappa + \delta \cdot g_{I,J} + r_\delta \) with \( \delta = \frac{1}{2} \) and \( r_\delta = 0 \).

Lastly, Figure \ref{fig: invarian densities symmetric} compares the invariant densities \( f_0 \) and \( f_{1/2} \) for the transfer operators \( \mathcal{L}_0 \) and \( \mathcal{L}_{1/2} \), respectively. These two symmetric densities, both normalized to satisfy \( \norm{f_0}_{L^1(D)} = \norm{f_{1/2}}_{L^1(D)} \), weight the points in \( D \) differently. Specifically, \( f_0 \) is concentrated around \( \pm 1 \), whereas \( f_{1/2} \), while preserving the same maximum points as \( f_0 \), also shows a third local maximum at \( x = 0 \), where the observable is concentrated.

\begin{figure}[!htb]
 \centering
 \subfloat[Observable.]{%
 \includegraphics[width=0.43\textwidth]{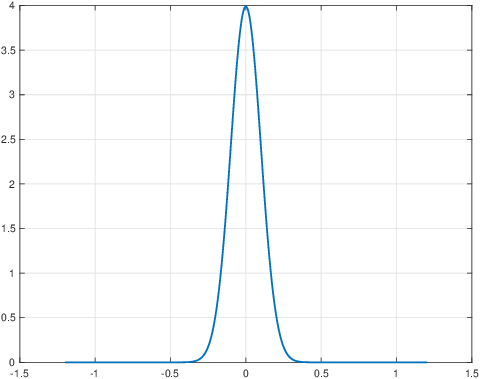}
       \label{fig: observable symmetric}
 }\quad
 \subfloat[Optimal perturbation $g_{I,J}$.]{%
 \includegraphics[width=0.5\textwidth]{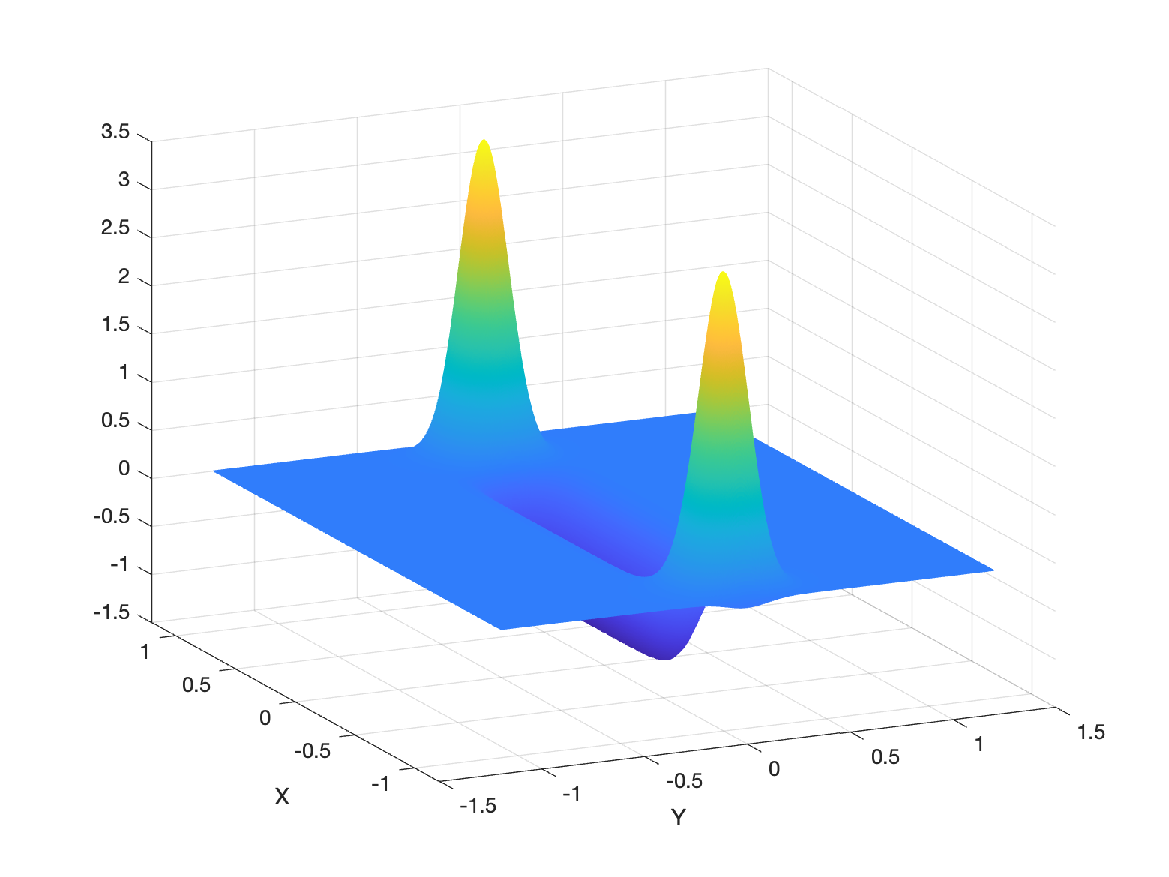}
            \label{fig: optimal perturbation symmetric}}
 \qquad
 \subfloat[Perturbed kernel $\kappa_{1/2} = \kappa + \frac{1}{2} g_{I,J}$.]{%
      \includegraphics[width=0.49\textwidth]{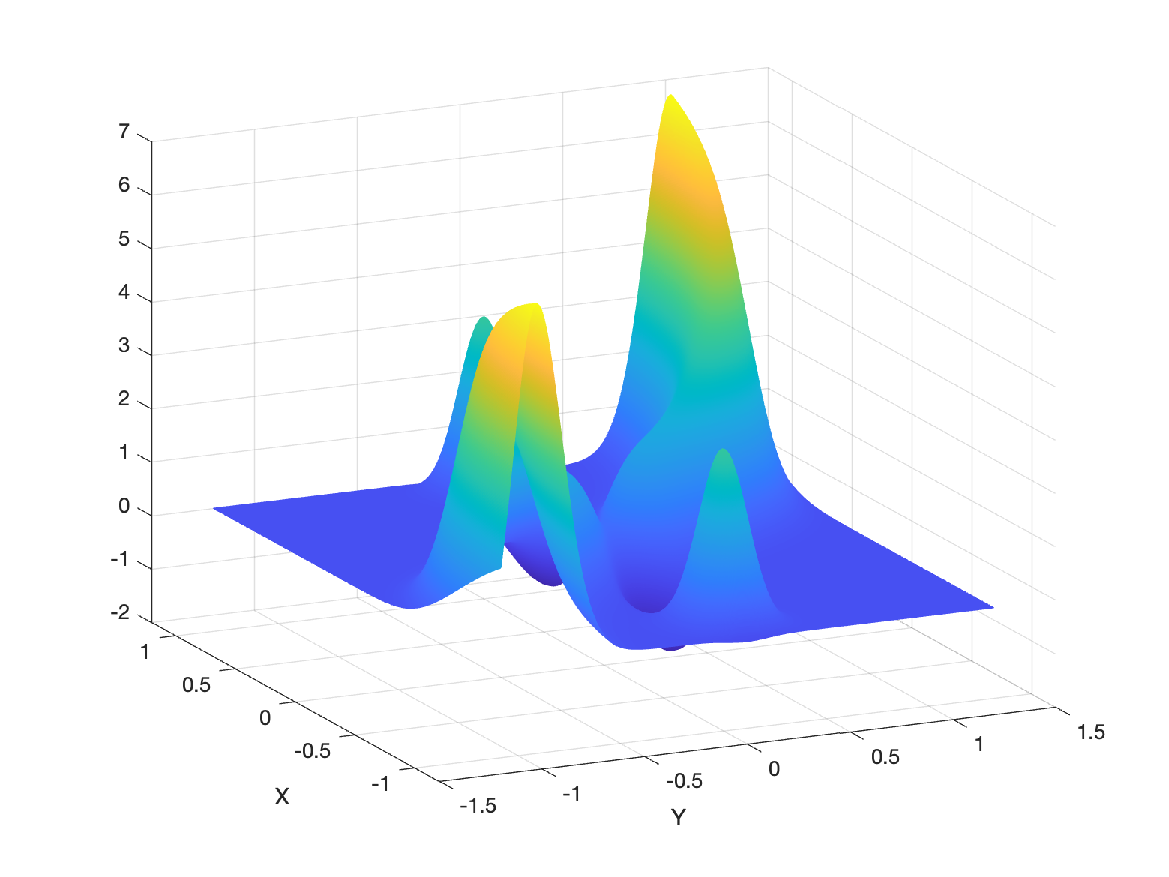}
            \label{fig: perturbed kernel symmetric}
      }
       \subfloat[Invariant densities $f_0$ and $f_{1/2}$.]{%
      \includegraphics[width=0.47\textwidth]{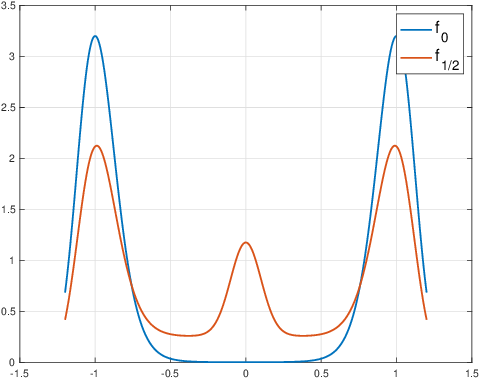}
            \label{fig: invarian densities symmetric}
      }
       \caption{Symmetric experiment, $\Delta x = \Delta t =  2 \cdot 10^{-3},$ $T = 1$, $D = [-1.2,1.2]$, $I = 35$, $J = 35$. (a) Symmetric observable $y \mapsto \phi(y)$. (b) Optimal perturbation $g_{I,J} = \frac{\sum_{r \, : h_r \in \mathcal{B}_{I,J}} G_r h_r}{   \left( \sum_{r \, : h_r \in \mathcal{B}_{I,J}} G_r^2 \right)^{1/2}} $. (c) Perturbed kernel $k_{1/2} = \kappa +\frac{1}{2} g_{I,J}$. (d) Invariant densities $f_0$ and $f_{1/2}$, for $\mathcal{L}_0$ and $\mathcal{L}_{1/2}$ respectively. }%
        \label{fig: symmetric experiment}
\end{figure}

\begin{figure}[!htb]
 \centering
 \subfloat[Observable.]{%
      \includegraphics[width=0.43\textwidth]{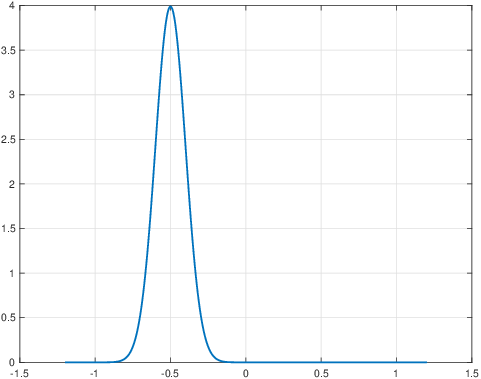}}
 \qquad
 \subfloat[Optimal perturbation $g_{I,J}$.]{%
      \includegraphics[width=0.49\textwidth]{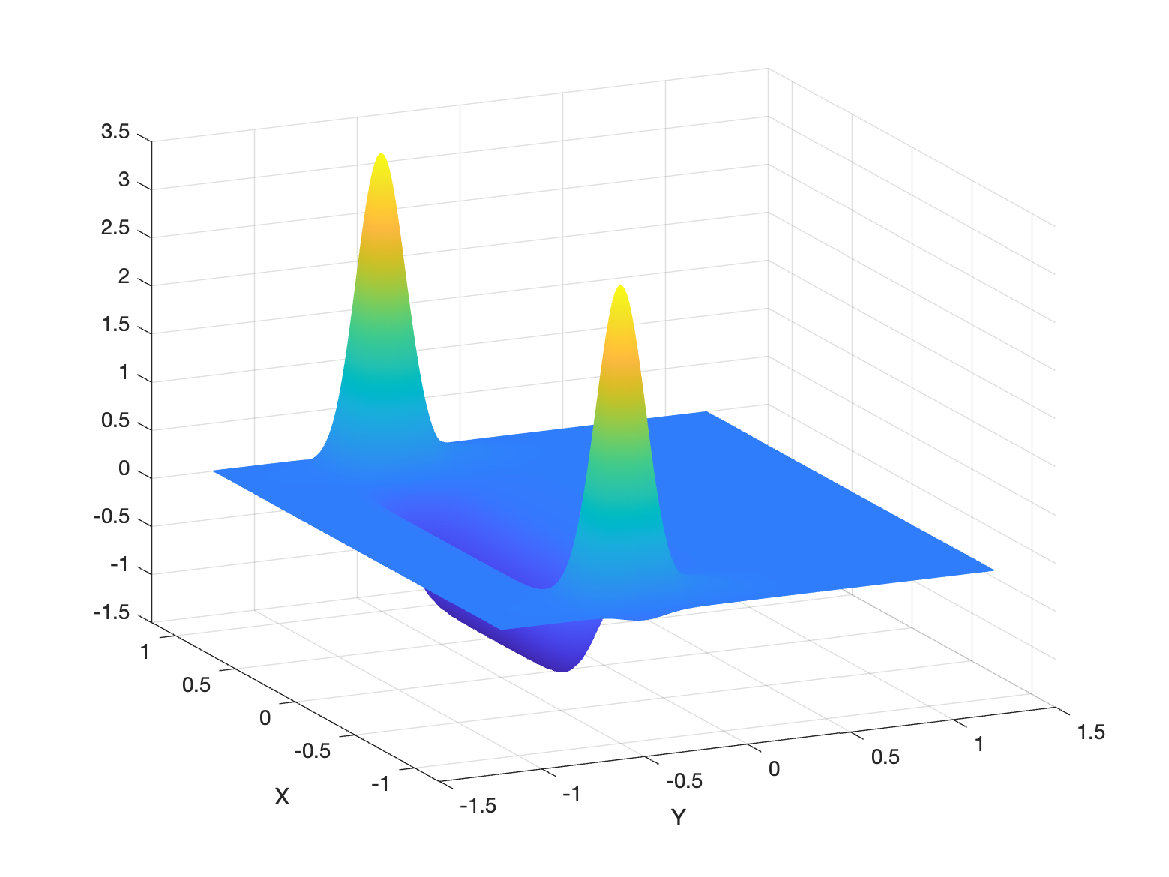}}
 \qquad
 \subfloat[Perturbed kernel $\kappa_{1/2} = \kappa + \frac{1}{2} g_{I,J}$.]{%
      \includegraphics[width=0.49\textwidth]{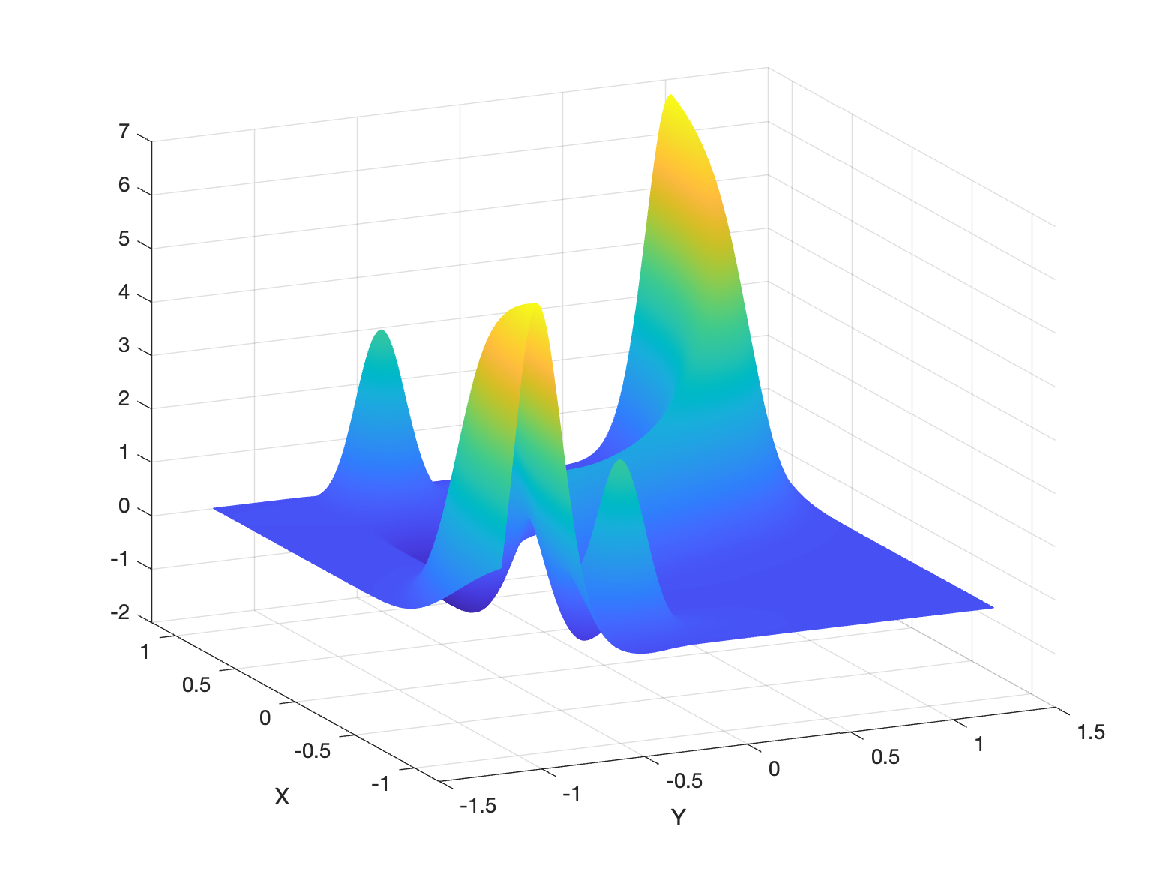}}
       \subfloat[Invariant densities $f_0$ and $f_{1/2}$.]{%
      \includegraphics[width=0.47\textwidth]{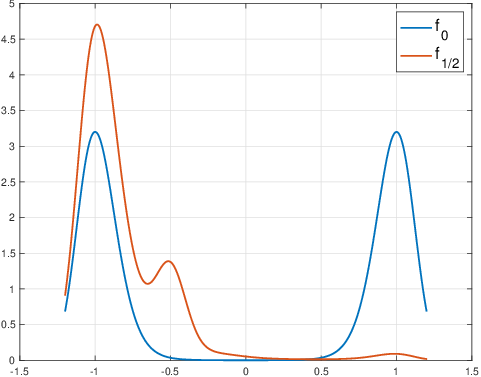}}
       \caption{Asymmetric experiment, $\Delta x = \Delta t =  2 \cdot 10^{-3},$ $T = 1$, $D = [-1.2,1.2]$, $I = 35$, $J = 35$. (a) Asymmetric observable $y \mapsto \phi(y)$. (b) Optimal perturbation $g_{I,J} = \frac{\sum_{r \, : h_r \in \mathcal{B}_{I,J}} G_r h_r}{   \left( \sum_{r \, : h_r \in \mathcal{B}_{I,J}} G_r^2 \right)^{1/2}} $. (c) Perturbed kernel $k_{1/2} = \kappa +\frac{1}{2} g_{I,J}$. (d) Invariant densities $f_0$ and $f_{1/2}$, for $\mathcal{L}_0$ and $\mathcal{L}_{1/2}$ respectively. }%
        \label{fig: asymmetric experiment}
\end{figure}

\section*{Code availability}
All material in the text and figures was produced by the authors using standard mathematical and numerical analysis tools. The only externally supplied code of this work consists of the implementation of the Simpson's rule for numerical integration by Damien Garcia (\href{https://www.mathworks.com/matlabcentral/fileexchange/25754-simpson-s-rule-for-numerical-integration}{link}), which we acknowledge.\\ The code for the numerical simulations performed in Section \ref{sec:numerics} is available at Zenodo (\href{https://doi.org/10.5281/zenodo.13820212}{https://doi.org/10.5281/zenodo.13820212}).\footnote{The code available on Zenodo performs the numerical experiments described in this work using $\Delta x = 8 \cdot 10^{-3}$ and $\Delta t = 2 \cdot 10^{-3}$. To reproduce the results of this work, i.e., with $\Delta x = 2 \cdot 10^{-3}$, the reader should modify the number of points in the spatial mesh at line 16 of the code.
}

\clearpage
\appendix

\renewcommand{\theequation}{A\arabic{equation}}
\setcounter{equation}{0}

\section{Recap of convex optimisation}\label{appendix}
In this section we recall some general result on the optimization of linear functions in convex sets adapted for our needs (see \cite{AFG22} for the proofs and other details).  
Let \(P\) be a bounded and convex subset of a Hilbert space \(\fH\).
\begin{definition}
\label{stconv}We say that a convex closed set $P\subseteq \fH$ is \emph{strictly convex} if for each pair $x,y\in P$ and for all $\gamma\in(0,1)$, the points $\gamma x+(1-\gamma)y\in \mathrm{int}(P)$, where \(\mathrm{int}(P)\) is the relative interior\footnote{The relative interior of a closed convex set $C$ is the interior of $C$ relative to the closed affine hull of $C$.} of \(P\).
\end{definition}

Let us briefly recall some relevant results from convex optimisation.\\
Suppose $\fH$ is a separable Hilbert space and $P\subset \fH$. Let $\mathcal{J}:\fH\rightarrow {\mathbb{R}}$ be a continuous linear
function. Consider the abstract problem to find $p_*\in P$ such that 
\begin{equation}
\mathcal{J}(p_*)=\max_{p\in P}\mathcal{J}(p) .  \label{gen-func-opt-prob}
\end{equation}%

The existence and uniqueness of an optimal perturbation follows from
properties of $P$ as stated in the following two propositions.

\begin{proposition}[Existence of the optimal solution]
\label{prop:exist} Let $P$ be bounded, convex, and closed in $\fH$.
Then problem~\eqref{gen-func-opt-prob} has at least one solution.
\end{proposition}

Upgrading convexity of the feasible set $P$ to strict convexity provides
uniqueness of the optimal solution.

\begin{proposition}[Uniqueness of the optimal solution]
\label{prop:uniqe} Suppose $P$ is closed, bounded, and strictly convex
subset of $\fH$, and that $P$ contains the zero vector in its
relative interior. If $\mathcal{J}$ is not uniformly vanishing on $P$, then the optimal solution to \eqref{gen-func-opt-prob} is unique.
\end{proposition}

Note that in the case when $\mathcal{J}$ is uniformly vanishing, all the elements of $P$ are solutions of the problem $\eqref{gen-func-opt-prob}.$

\subsection*{Declaration}

\subsubsection*{Acknowledgement}
The authors thank Carlangelo Liverani, Warwick Tucker, Andy Hammerlindl, Franco Flandoli, and Silvia Morlacchi for valuable discussions/suggestions. We also thank Caroline Wormell and Edoardo Lombardo for their help with numerical algorithms initially. G.D.S. was supported by the Italian national interuniversity PhD course in Sustainable Development and Climate Change for a part of the project, and acknowledges the current support of DFG project FOR 5528. S.J. was supported by  Università degli studi di Roma Tor Vergata for a part of this project, afterwards by ARC (DP220100492) and Monash University. S.G. and S.J. were funded by the PRIN grant “Stochastic properties of dynamical systems" (PRIN 2022NTKXCX), and extend their thanks to Centro de Giorgi and Scuola Normale Superiore for facilitating through a part of this project.

\subsubsection*{Conflict of interest} The authors have no relevant financial or non-financial conflict of interests to disclose.

\subsubsection*{Data availability} Data sharing not applicable to this article as no datasets were generated or analysed during the current study.
\clearpage
\bibliographystyle{plain}
 \bibliography{biblio}

\end{document}